\documentclass{article}

\usepackage{graphicx}

\usepackage{authblk}

\usepackage{tikz}
\usetikzlibrary{arrows,calc}

\tikzstyle{edge}=[shorten <=2pt, shorten >=2pt, >=stealth, line width=0.8pt]
\tikzstyle{arc}=[->, shorten <=2pt, shorten >=2pt, >=stealth, line width=0.8pt]
\tikzstyle{bentE}=[shorten <=1pt, shorten >=1pt, >=stealth, bend right=30,
                   bend left=30, line width=1pt]
\tikzstyle{vertex}=[circle, fill=white, draw, minimum size=5pt, inner sep=0pt,
                    outer sep=0pt]
\tikzstyle{blackV}=[circle, fill=black, minimum size=5pt, inner sep=0pt,
                    outer sep=0pt]

\usepackage{float}

\usepackage{xcolor}

\usepackage{amsmath}
\usepackage{amssymb}

\usepackage{amsthm}

\usepackage[hidelinks]{hyperref}

\usepackage[capitalize]{cleveref}

\newtheorem{theorem}{Theorem}
\newtheorem{lemma}[theorem]{Lemma}
\newtheorem{corollary}[theorem]{Corollary}
\newtheorem{proposition}[theorem]{Proposition}
\newtheorem{remark}[theorem]{Remark}
\newtheorem{problem}{Problem}

\title{Minimal obstructions for polarity, monopolarity, unipolarity and
$(s,1)$-polarity in generalizations of cographs%
\thanks{The authors gratefully acknowledge support from grants SEP-CONACYT
A1-S-8397, DGAPA-PAPIIT IA104521, and CONACYT FORDECYT-PRONACES/39570/2020}}

\author[1]{Fernando Esteban Contreras-Mendoza\thanks{
	esteban.contreras@ciencias.unam.mx}}
\author[1]{C\'esar Hern\'andez-Cruz\thanks{chc@ciencias.unam.mx}}

\affil[1]{Facultad de Ciencias\\
Universidad Nacional Aut\'onoma de M\'exico\\
Av. Universidad 3000, Circuito Exterior S/N\\
C.P. 04510, Ciudad Universitaria, CDMX, M\'exico}

\begin{document}
\date{}

\maketitle
\begin{abstract}
  It is known that every hereditary property can be characterized by finitely
  many minimal obstructions when restricted to either the class of cographs or
  the class of $P_4$-reducible graphs.   In this work, we prove that also when
  restricted to the classes of $P_4$-sparse graphs and $P_4$-extendible graphs
  (both of which extend $P_4$-reducible graphs) every hereditary property can be
  characterized by finitely many minimal obstructions.

  We present complete lists of $P_4$-sparse and $P_4$-extendible minimal
  obstructions for polarity, monopolarity, unipolarity, and $(s,1)$-polarity,
  where $s$ is a positive integer.   In parallel to the case of $P_4$-reducible
  graphs, all the $P_4$-sparse minimal obstructions for these hereditary
  properties are cographs.
\end{abstract}

\section{Introduction}

All graphs in this paper are simple and finite; we refer the reader to
\cite{bondy2008} for basic terminology and notation not explicitly defined. For
a graph $G$, we denote its vertex set by $V_G$, and its complement by
$\overline{G}$. As usual, we use $N_G(v)$ to denote the neighborhood of the
vertex $v$ in $G$, and $d_G(v)$ to denote the degree of $v$ in $G$, i.e. $d_G(v)
= |N_v(G)|$. If no confusion is possible we write $N(v)$ and $d(v)$ instead of
$N_G(v)$ and $d_G(v)$, respectively. For graphs $G$ and $H$, we denote by $G +
H$ the disjoint union of $G$ and $H$, and by $G \oplus H$ the join of $G$ and
$H$, i.e., the graph $\overline{\overline{G} + \overline{H}}$. Naturally, the
disjoint union of $n$ copies of a graph $G$ is denoted by $nG$. Two subsets
$V_1$ and $V_2$ of $V_G$ are said to be {\em completely adjacent} if every
vertex of $V_1$ is adjacent to any vertex of $V_2$. Analogously, $V_1$ is {\em
completely nonadjacent} to $V_2$ if no vertex in $V_1$ is adjacent to a vertex
in $V_2$.

If $G$ and $H$ are graphs, we write $H \le G$ to denote that $H$ is an induced
subgraph of $G$.  We say that $G$ is an {\em $H$-free graph} if $G$ does not
have induced subgraphs isomorphic to $H$. Given a family of graphs
$\mathcal{H}$, we say that $G$ is {\em $\mathcal{H}$-free} if it is $H$-free for
every $H \in \mathcal{H}$. A {\em $k$-cluster} is the complement of a complete
$k$-partite graph, and a {\em cluster} is a $k$-cluster for some integer $k$.
Clusters are characterized as $P_3$-free graphs, while $k$-clusters are
precisely the $\{P_3, (k+1)K_1\}$-free graphs.  Complementarily, complete
multipartite graphs are precisely $\overline{P_3}$-free graphs, and complete
$s$-partite graphs are precisely $\{\overline{P_3}, K_{s+1}\}$-free graphs.

A property $\mathcal{P}$ of graphs is said to be {\em hereditary} if it is
closed under taking induced subgraphs. Given a hereditary property of graphs
$\mathcal P$, a {\em $\mathcal{P}$-obstruction} is a graph $G$ that does not
have the property $\mathcal{P}$; if in addition $G$ is such that any proper
induced subgraph of $G$ has the property $\mathcal P$, then $G$ is said to be a
{\em minimal $\mathcal{P}$-obstruction}.

For nonnegative integers $s$ and $k$, an {\em $(s, k)$-polar partition} of a
graph $G$ is a partition $(A, B)$ of $V$ such that $G[A]$ is a complete
multipartite graph with at most $s$ parts and $G[B]$ is a $k$-cluster. A graph
$G$ is said to be {\em $(s, k)$-polar} if $V$ admits an $(s, k)$-polar
partition.   When we replace $s$ or $k$ with $\infty$, it means that the number
of parts of $G[A]$ or $G[B]$, respectively, is unbounded. A graph is said to be
{\em monopolar} or {\em polar} if it is a $(1, \infty)$- or an $(\infty,
\infty)$-polar graph, respectively. A {\em unipolar partition} of a graph $G$ is
a polar partition $(A, B)$ of $G$ such that $A$ is a clique.  Naturally, a graph
is said to be {\em unipolar} if it admits a unipolar partition. Unipolar and
monopolar graphs are particularly interesting because many recognition
algorithms for polar graphs on specific graph classes first check whether the
input graph is either unipolar or monopolar.

In the much larger context of matrix partitions, it was shown that for any pair
of fixed nonnegative integers, $s$ and $k$, there are only finitely many minimal
$(s,k)$-polar obstructions \cite{federENDM28}, and therefore the class of
$(s,k)$-polar graphs can be recognized by a brute force algorithm in polynomial
time. Also, unipolar graphs have been shown to be efficiently recognizable
\cite{churchleyGC30,eschenDAM162}. In contrast, the problems of deciding whether
a graph is polar and deciding whether a graph is monopolar have been shown to be
NP-complete \cite{chernyakDM62,farrugiaTEJC11} even when restricted to
triangle-free planar graphs \cite{leISAC}.   Such results encouraged the study
of polarity and monopolarity in many graph classes as cographs
\cite{ekimDAM156b}, chordal graphs \cite{ekimDAM156a,stachoDD}, permutation
graphs \cite{ekimIWCA,ekimEJC34}, trivially perfect graphs
\cite{talmaciuIJCCC5}, line graphs \cite{churchleySIAMJDM25}, triangle-free and
claw-free graphs \cite{churchleyDM312,churchleyJGT76}, comparability graphs
\cite{churchleyGC30}, planar graphs \cite{leTCS528,leISAC}, and line-graphs of
bipartite graphs \cite{ekimDAM158}, just to mention some of the most outstanding
classes.

Cographs were introduced in \cite{corneilDAM3}, where it was proved that such
graphs are precisely the $P_4$-free graphs, and also the graphs that can be
obtained from trivial graphs by disjoin union and join operations. Thus, for any
nontrivial cograph $G$, either $G$ or $\overline G$ is disconnected.  A rooted
labeled tree uniquely representing $G$ (the {\em cotree} of $G$) can be
constructed in $O(|V|+|E|)$-time using LexBFS \cite{bretscherSIAMJDM22}.  It
follows from the uniqueness of the cotree representation of a cograph that many
algorithmic problems which are difficult for general graphs can be efficiently
solved on cographs by using its cotree \cite{corneilDAM3}. Additionally,
cographs inherit efficient algorithms for some superclasses they belong to, like
distance hereditary, permutation, and comparability graphs.

Cographs possess many desirable structural properties, and they are also
particularly interesting since real-life applications quite often involve graph
models where paths of length four are unlikely to appear \cite{corneilSIAMJC14}.
For the above reasons, the study of cographs was naturally followed by the
introduction of a wide variety of cograph superclasses having both few induced
$P_4$'s and a unique tree representation.  For instance, a graph is said to be a
{\em $P_4$-sparse graph} if any set of five vertices induces at most one $P_4$,
and a {\em $P_4$-extendible graph} is a graph such that, for any vertex subset
$W$ inducing a $P_4$, there exists at most one vertex $v \notin W$ which belongs
to a $P_4$ sharing vertices with $W$.

Ekim, Mahadev and de Werra found the complete list of cograph minimal polar
obstructions as well as the exact list of cograph minimal $(s, k)$-polar
obstructions when $\min\{s,k\} = 1$ \cite{ekimDAM156b}. In the past few years,
the study of $(s,k)$-polarity in cographs has continued with the following main
results.   In \cite{hellDAM261}, Hell, Hern\'andez-Cruz and Linhares-Sales
provided a full characterization of cograph minimal $(2,2)$-polar obstructions.
Bravo, Nogueira, Protti and Vianna exhibited the exhaustive list of cograph
minimal $(2,1)$-polar obstructions \cite{bravoAETCI}, and Contreras-Mendoza and
Hern\'andez-Cruz proved a simple recursive characterization for all the cograph
minimal $(s, 1)$-polar obstructions for any arbitrary integer $s$, as well as
the complete list of cograph minimal monopolar obstructions
\cite{contrerasDAM281}. The authors of the present work provided in
\cite{contrerasDM7} complete lists of cograph minimal $(\infty, k)$-polar
obstructions for $k = 2$ and $k = 3$, as well as a partial recursive
characterization for arbitrary values of $k$.

In this paper we study $(s, k)$-polarity on two cograph superclasses, namely
$P_4$-sparse and $P_4$-extendible graphs. We give complete lists of minimal $(s,
1)$-polar obstructions, minimal unipolar obstructions, minimal monopolar
obstructions, and minimal polar obstructions on such cograph superclasses.
Additionally, we prove that any hereditary property has only finitely many
minimal obstructions in the classes of $P_4$-sparse and $P_4$-extendible graphs.
A summary of our main results is included in \Cref{sec:conclusions}.

The rest of the paper is organized as follows. In \Cref{sec:cogGen}, we review
some families of graphs generalizing cographs, emphasizing $P_4$-sparse and
$P_4$-extendible graphs by their relevance in this paper. \Cref{sec:wqo} is
devoted to prove that any hereditary property has finitely many minimal
obstructions when restricted to $P_4$-sparse or $P_4$-extendible graphs.
Complete lists of $P_4$-sparse and $P_4$-extendible minimal unipolar
obstructions are given in \Cref{sec:MinUnipObs}. In
\Cref{sec:discM(s1)PO,sec:connectedM(s1)PO} we provide complete lists of
disconnected minimal $(s,1)$- and $(\infty, 1)$-polar obstructions for general
graphs, as well as technical results we need to characterize connected
$P_4$-sparse and $P_4$-extendible minimal $(s,1)$- and $(\infty, 1)$-polar
obstructions. Finally, in \Cref{sec:main} we prove our main results about
polarity on cograph generalizations: we give complete characterizations for
$P_4$-sparse and $P_4$-extendible minimal $(s,1)$-, $(\infty,1)$-, and $(\infty,
\infty)$-polar obstructions. Conclusions, work in progress, and some open
problems and conjectures are presented in \Cref{sec:conclusions}.


\section{Two families generalizing cographs}
\label{sec:cogGen}

In \cite{corneilDAM3}, eight characterizations of cographs were presented; the
most relevant for this work are summarized in the following proposition.

\begin{theorem}[\cite{corneilDAM3}]
\label{theo:cographs}
Let $G$ be a graph. The following statements are equivalent.
\begin{enumerate}
  \item $G$ can be constructed from trivial graphs by means of disjoint union
    and complement operations (it is a cograph).

  \item $G$ can be constructed from trivial graphs by means of join and disjoint
    union operations.

  \item $G$ is a $P_4$-free graph.

  \item For any nontrivial induced subgraph $H$ of $G$, either $H$ or
    $\overline{H}$ is disconnected.
\end{enumerate}
\end{theorem}

Clearly, the family of cographs is an auto-complementary and hereditary class of
graphs. Moreover, it follows from the structural characterizations of cographs
that they can be uniquely represented by a rooted labeled tree, its cotree
\cite{corneilDAM3}. Cographs can be recognized and their cotree can be
constructed in linear time by a certifying LexBFS algorithm
\cite{bretscherSIAMJDM22}. In addition, many algorithmic problems which are
difficult on general graphs can be efficiently solved on cographs using
bottom-up algorithms on their cotrees.

Some real life applications involve graph models which are unlikely to have more
than some few induced paths on four vertices \cite{corneilSIAMJC14}. From this
point of view, cographs are the most restrictive class ($P_4$-free), so a
natural question is whether some cograph superclass with weaker restrictions on
the amount of induced $P_4$'s has similar properties, i.e., allows us to develop
efficient algorithms for problems which are difficult in general graphs. Below,
we introduce some graph classes with few induced paths of length three, which
have the property of having a constructive characterization from simple
primitive graphs and using simple graph operations. Such characterizations imply
that these graph classes can be recognized in linear time and a tree
representation (similar to the cotree) can be efficiently computed. Before
introducing such families, we give some necessary definitions.

A $(1,1)$-polar partition of a graph $G$ is commonly called a {\em split
partition} of $G$. The graphs admitting a split partition are the {\em split
graphs} and they are characterized as the $\{2K_2, C_4, C_5\}$-free graphs
\cite{foldesGTC1977}. The $\{2K_2, C_4\}$-free graphs are known as {\em
pseudo-split graphs}. A graph $G$ of order at least four is said to be a {\em
headless spider} if there exist a split partition $(S, K)$ of $V$ and a
bijection $f \colon S \to K$ such that either $N(s) = \{f(s)\}$ for any $s \in
S$, or $N(s) = K \setminus \{f(s)\}$ for every $s \in S$. Given a graph $G$ and
an induced path $P$ of length three, a {\em partner} of $P$ is a vertex $v$ of
$G$ such that $V_P \cup \{v\}$ induces some of $C_5, P_5, P, F$ or its
complements (see \Cref{fig:extSets}). Now, we introduce some cograph
generalizations.

A graph $G$ is said to be:
\begin{enumerate}
  \item {\em $P_4$-reducible} if any vertex belongs to at most one induced
    $P_4$.

  \item A {\em $(q,t)$-graph} if no set of at most $q$ vertices induces more
    than $t$ distinct $P_4$'s. The $(5,1)$-graphs are called {\em $P_4$-sparse
    graphs}.

  \item {\em Extended $P_4$-reducible} if both, $G$ and $\overline{G}$, are
    $\{P_5, F, P, \textnormal{net}\}$-free graphs (see \Cref{fig:net}).

  \item {\em Extended $P_4$-sparse} if both, $G$ and $\overline{G}$, are
    $\{P_5, F, P\}$-free graphs.

  \item {\em $P_4$-lite} if every induced subgraph of order at most six is
    either isomorphic to a headless spider, or it contains at most two induced
    $P_4$'s.

  \item {\em $P_4$-extendible} if for any vertex subset $W$ inducing a $P_4$,
    there exists at most one vertex $v \notin W$ which belongs to a $P_4$
    sharing vertices with $W$.

  \item {\em $P_4$-tidy} if any induced $P_4$ has at most one partner.

  \item {\em $P_4$-laden} if any induced subgraph of $G$ of order at most six
    either is a split graph or it contains at most two induced $P_4$'s.

  \item {\em Extended $P_4$-laden} if any induced subgraph of $G$ of order at
    most six either is a pseudo-split graph or it contains at most two induced
    $P_4$'s.
\end{enumerate}

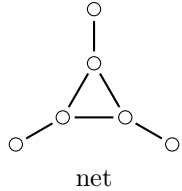
\begin{figure}[ht!]
\begin{center}
\begin{tikzpicture}

\begin{scope}[scale=0.6]
  \foreach \i in {0,1,2}
    \foreach \j in {1,2}
      \node [vertex] (\i\j) at ({90+(\i*120)}:{(\j*1.2)-0.4})[]{};

  \foreach \i/\j in {01/11,11/21,21/01,01/02,11/12,21/22}
    \draw [edge] (\i) to (\j);

  \node (n) at (0,-1.75){net};
\end{scope}

\end{tikzpicture}
\end{center}
\caption{The net graph.}
\label{fig:net}
\end{figure}

It is worth emphasizing that we are not really interested in general
$(q,t)$-graphs, but in $(q,q-4)$-graphs. This is due to, for any fixed $q$, $(q,
q-4)$-graphs are known to have simple enough tree representations, while it
remains unknown whether $(q,t)$-graphs have such representations for arbitrary
values of $q$ and $t$.

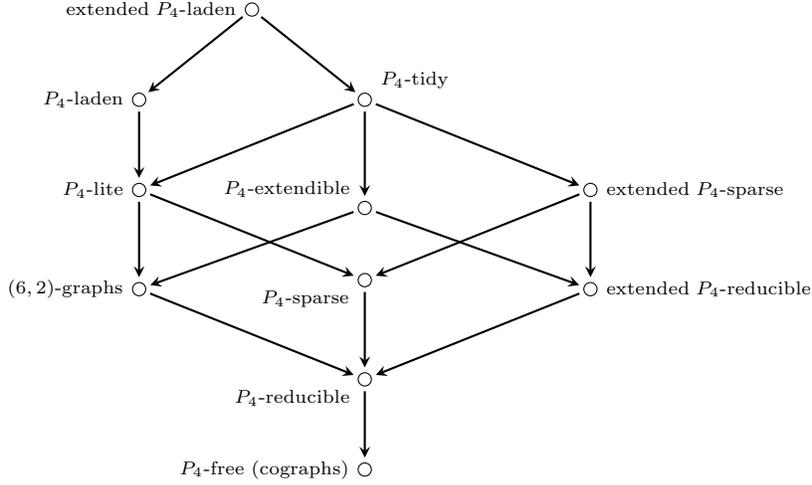
\begin{figure}[ht!]
\begin{center}
\begin{tikzpicture}

\begin{scope}[xscale= 1.5, yscale=1.2]
  \node [vertex] (exP4la) at (-1,5.1)[label=180:\scriptsize{extended
                                                            $P_4$-laden}]{};
  \node [vertex] (P4la) at (-2,4.1)[label=180:\scriptsize{$P_4$-laden}]{};
  \node [vertex] (P4ti) at (0,4.1)[label=15:\scriptsize{$P_4$-tidy}]{};
  \node [vertex] (P4ex) at (0,2.9)[label=165:\scriptsize{$P_4$-extendible}]{};
  \node [vertex] (exP4sp) at (2,3.1)[label=0:\scriptsize{extended
                                                          $P_4$-sparse}]{};
  \node [vertex] (P4li) at (-2,3.1)[label=180:\scriptsize{$P_4$-lite}]{};
  \node [vertex] (62) at (-2,2)[label=180:\scriptsize{$(6, 2)$-graphs}]{};
  \node [vertex] (P4sp) at (0,2.1)[label=195:\scriptsize{$P_4$-sparse}]{};
  \node [vertex] (exP4re) at (2,2)[label=0:\scriptsize{extended
                                                        $P_4$-reducible}]{};
  \node [vertex] (P4re) at (0,1)[label=195:\scriptsize{$P_4$-reducible}]{};
  \node [vertex] (cog) at (0,0)[label=180:\scriptsize{$P_4$-free (cographs)}]{};

  \foreach \i/\j in {exP4la/P4la, exP4la/P4ti, P4la/P4li, P4ti/P4ex,
                      P4ti/exP4sp, P4ti/P4li, P4ex/exP4re, P4ex/62, exP4sp/P4sp,
                      exP4sp/exP4re, P4li/62, P4li/P4sp, 62/P4re, P4sp/P4re,
                      exP4re/P4re, P4re/cog}
    \draw [arc] (\i) to (\j);
\end{scope}

\end{tikzpicture}
\end{center}
\caption{Relations between some graph classes with just a few induced $P_4$'s.}
\label{fig:relations}
\end{figure}

Giakoumakis and Vanherpe \cite{giakoumakisTCS180} observed that $P_4$-reducible
graphs are the graphs which are both $P_4$-sparse and $P_4$-extendible graphs.
Some other relations between the graph classes introduced above can be
established from their definitions or using diverse characterizations for them.
We represent the containment relationships between these classes in
\Cref{fig:relations}, where an arc from a class $\mathcal G$ to a class
$\mathcal H$ means that $\mathcal H \subseteq \mathcal G$. We remark that any
graph class represented in \Cref{fig:relations} can be recognized, and a tree
representation can be obtained, in polynomial time.   Moreover, in most cases
this can be done in linear time.

In the following sections we state constructive characterizations for
$P_4$-sparse and $P_4$-extendible graphs which will be the cornerstones for most
of the work in this paper.


\subsection{$P_4$-sparse graphs}
\label{sec:P4sparse}

A graph is defined to be $P_4$-sparse if any vertex subset with at most five
vertices induces at most one $P_4$. It follows directly from this definition
that a graph $G$ is a $P_4$-sparse graph if and only if both, $G$ and
$\overline{G}$, are $\{C_5, P_5, P, F\}$-free graphs (see \Cref{fig:extSets}).
Additionally, Jamison and Olariu gave a connectedness characterization of
$P_4$-sparse graphs based on spiders which we now introduce.

A graph $G$ is a {\em spider} if its vertex set can be partitioned into into $S,
K$ and $R$ in such a way that $G[S \cup K]$ is a headless spider with partition
$(S, K)$, $R$ is completely adjacent to $K$ and completely nonadjacent to $S$.
For a spider $G = (S, K, R)$ we say that $S$ is its {\em legs set}, $K$ is its
{\em body}, and $R$ is its {\em head}. A spider is called {\em thin}
(respectively {\em thick}) if $d(s) = 1$ (respectively $d(s) = |K|-1$) for any
$s \in S$. Notice that the complement of a thin spider is a thick spider, and
vice versa, and that a headless spider is precisely a spider with an empty head.

\begin{theorem}[\cite{jamisonDAM35}]
\label{thm:connCharSparse}
Let $G$ be a graph. Then, $G$ is a $P_4$-sparse graph if and only if for every
nontrivial induced subgraph $H$ of $G$, exactly one of the following statements
is satisfied
\begin{enumerate}
	\item $H$ is disconnected.

	\item $\overline{H}$ is disconnected.

	\item $H$ is a spider.
\end{enumerate}
\end{theorem}


\subsection{$P_4$-extendible graphs}

Let $G$ be a graph, and let $W$ be a proper subset of its vertex set. Denote by
$S(W)$ the set of vertices $x \in V - W$ such that $x$ belongs to a $P_4$
sharing vertices with $W$. In case $S(W)$ contains at most one vertex, we will
say that $W$ has a {\em proper extension}. A set $D$ of vertices is said to be
an {\em extension set} if $D = W \cup S(W)$ for a set $W$ which induces a $P_4$
and has a proper extension. An extension set $D$ is said to be {\em separable}
if no vertex of $D$ is both an endpoint of some $P_4$ in $G[D]$ and a midpoint
of some $P_4$.

In the above terms, $P_4$-extendible graphs are the graphs such that every set
inducing a $P_4$ has a proper extension. As Jamison and Olariu noticed in
\cite{jamisonDAM34}, any extension set must induce one of the eight graphs
depicted in \Cref{fig:extSets}, namely $P_4, C_5, P_5, P, F$ or their
complements. We call these graphs {\em extension graphs}. In addition, separable
extension sets must induce one of $P_4,P,F$ or their complements. These graphs
are called {\em separable extension graphs}.

\begin{figure}[ht!]
\centering
\begin{tikzpicture}
\begin{scope}[scale=0.85]

\begin{scope}[xscale=0.8] 
\foreach \i in {0,3}
	\node [vertex] (\i) at (\i,0)[]{};
\foreach \i in {1,2}
	\node [blackV] (\i) at (\i,0)[]{};

\foreach \i in {0,1,2}
	\draw let \n1={int(\i+1)} in [edge]
		(\i) to node [above] {} (\n1);
\node [rectangle] (n) at (1.5,-0.8){$P_4$};
\end{scope}

\begin{scope}[xshift=4cm, xscale=0.8] 
\foreach \i in {0,1,2,3}
	\node [vertex] (\i) at (\i,0)[]{};
\node [vertex] (4) at (1.5,1)[]{};

\foreach \i in {0,1,2}
	\draw let \n1={int(\i+1)} in [edge]
		(\i) to node [above] {} (\n1);

\foreach \i/\j in {4/0,4/3}
	\draw [edge] (\i) to node [above] {} (\j);
\node [rectangle] (n) at (1.5,-0.8){$C_5$};
\end{scope}

\begin{scope}[xshift=8cm, xscale=0.8] 
\foreach \i in {0,1,2,3}
	\node [vertex] (\i) at (\i,0)[]{};
\node [vertex] (4) at (0,1)[]{};

\foreach \i in {0,1,2}
	\draw let \n1={int(\i+1)} in [edge]
		(\i) to node [above] {} (\n1);

\foreach \i/\j in {4/0}
	\draw [edge] (\i) to node [above] {} (\j);
\node [rectangle] (n) at (1.5,-0.8){$P_5$};
\end{scope}

\begin{scope}[xshift=12cm, xscale=0.8] 
\foreach \i in {0,1,2,3}
	\node [vertex] (\i) at (\i,0)[]{};
\node [vertex] (4) at (2,1)[]{};

\foreach \i in {0,1,2}
	\draw let \n1={int(\i+1)} in [edge]
		(\i) to node [above] {} (\n1);

\foreach \i/\j in {4/0,4/2,4/3}
	\draw [edge] (\i) to node [above] {} (\j);
\node [rectangle] (n) at (1.5,-0.8){$\overline{P_5}$ (house)};
\end{scope}

\begin{scope}[yshift=-3cm, xscale=0.8] 
\foreach \i in {0,3}
	\node [vertex] (\i) at (\i,0)[]{};
\foreach \i in {1,2}
	\node [blackV] (\i) at (\i,0)[]{};
\node [blackV] (4) at (1,1)[]{};

\foreach \i in {0,1,2}
	\draw let \n1={int(\i+1)} in [edge]
		(\i) to node [above] {} (\n1);

\foreach \i/\j in {4/0,4/2}
	\draw [edge] (\i) to node [above] {} (\j);
\node [rectangle] (n) at (1.5,-0.8){$P$ (banner)};
\end{scope}

\begin{scope}[yshift=-3cm, xshift=4cm, xscale=0.8] 
\foreach \i in {0,3}
	\node [vertex] (\i) at (\i,0)[]{};
\foreach \i in {1,2}
	\node [blackV] (\i) at (\i,0)[]{};
\node [vertex] (4) at (0.5,1)[]{};

\foreach \i in {0,1,2}
	\draw let \n1={int(\i+1)} in [edge]
		(\i) to node [above] {} (\n1);

\foreach \i/\j in {4/0,4/1}
	\draw [edge] (\i) to node [above] {} (\j);

\node [rectangle] (n) at (1.5,-0.8){$\overline{P}$};
\end{scope}

\begin{scope}[yshift=-3cm, xshift=8cm, xscale=0.8] 
\foreach \i in {0,3}
	\node [vertex] (\i) at (\i,0)[]{};
\foreach \i in {1,2}
	\node [blackV] (\i) at (\i,0)[]{};
\node [vertex] (4) at (1,1)[]{};

\foreach \i in {0,1,2}
	\draw let \n1={int(\i+1)} in [edge]
		(\i) to node [above] {} (\n1);

\foreach \i/\j in {4/1}
	\draw [edge] (\i) to node [above] {} (\j);
\node [rectangle] (n) at (1.5,-0.8){$F$ (fork, chair)};
\end{scope}

\begin{scope}[yshift=-3cm, xshift=12cm, xscale=0.8] 
\foreach \i in {0,3}
	\node [vertex] (\i) at (\i,0)[]{};
\foreach \i in {1,2}
	\node [blackV] (\i) at (\i,0)[]{};
\node [blackV] (4) at (1,1)[]{};

\foreach \i in {0,1,2}
	\draw let \n1={int(\i+1)} in [edge]
		(\i) to node [above] {} (\n1);

\foreach \i/\j in {4/0,4/1,4/2}
	\draw [edge] (\i) to node [above] {} (\j);
\node [rectangle] (n) at (1.5,-0.8){$\overline{F}$ (kite)};
\end{scope}

\end{scope}
\end{tikzpicture}
\caption{The eight extension graphs. Black vertices are the midpoints of
  separable extension graphs.}
\label{fig:extSets}
\end{figure}
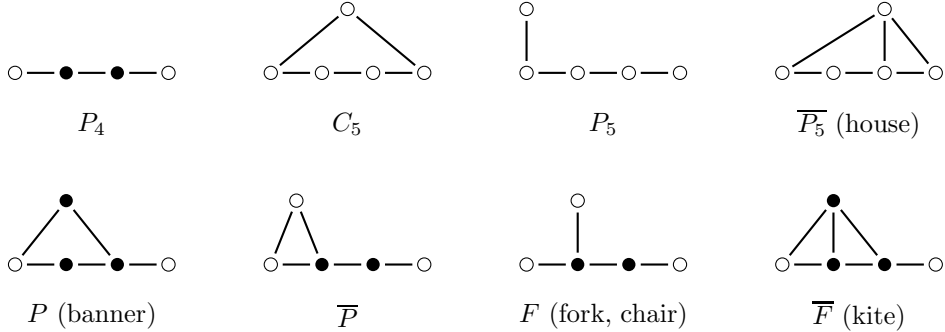

Let $G$ be a separable extension graph. We define a {\em $G$-spider} as an
(induced) supergraph $H$ of $G$ such that $V_H - V_G$ (denoted $R$) is
completely adjacent to the midpoints set of $G$ (denoted $K$), and $V_H - V_G$
is completely nonadjacent to the endpoints set of $G$ (denoted $S$). If $H$ is a
$G$-spider, we say that $(S, K, R)$ is a {\em $G$-spider partition} of $H$, and
we refer to $K$, $S$ and $R$ as the {\em body}, the {\em legs set}, and the {\em
head} of $H$, respectively. Along this text we will say that a given graph is an
$S$-spider if it is a $G$-spider for some arbitrary separable extension graph
$G$.

Observe that every extension graph is trivially a $P_4$-extendible graph but the
headless spiders on six vertices are examples of minimal $P_4$-extendible
obstructions. Thus, since any headless spider is a $P_4$-sparse graph and all
the forbidden $P_4$-sparse graphs are $P_4$-extendible, the classes of
$P_4$-sparse graphs and $P_4$-extendible graphs are incomparable.

Jamison and Olariu gave the following connectedness characterization for the
class of $P_4$-extendible graphs in \cite{jamisonDAM34}.

\begin{theorem}[\cite{jamisonDAM34}]
\label{thm:connChar}
If $G$ is a graph, then $G$ is a $P_4$-extendible graph if and only if, for
every nontrivial induced subgraph $H$ of $G$, precisely one of the following
conditions is satisfied:
\begin{enumerate}
	\item $H$ is disconnected.

	\item $\overline{H}$ is disconnected.

	\item $H$ is an extension graph.

	\item There is a unique separable extension graph $G$ such that $H$ is a
  	$G$-spider with nonempty head.
  \end{enumerate}
\end{theorem}

The class of $(6,2)$-graphs clearly is another superclass of cographs which is
incomparable with $P_4$-sparse graphs. As the following result states,
$P_4$-extendible graphs are closely related to $(6,2)$-graphs. With the help of
this theorem, analogous results for $(6,2)$-graphs can be obtained for each
proposition about $P_4$-extendible graphs. For the sake of length, we will not
explicitly state such results.

\begin{theorem}[\cite{babelDAM95}]
\label{thm:(6-2)vsP4ext}
If $G$ is a graph, then $G$ is a $(6,2)$-graph if and only if $G$ is a
$C_5$-free $P_4$-extendible graph.
\end{theorem}


\section{Well-quasi-orderings}
\label{sec:wqo}

Throughout this section, we show that any hereditary property has a finite
number of minimal obstructions when restricted to $P_4$-sparse or
$P_4$-extendible graphs. We will often use the following observation in the rest
of the text without mentioning it explicitly.

\begin{remark}
\label{rem:hered}
Let $\mathcal{P}$ be a hereditary property of graphs, and let $H$ be a
$\mathcal{P}$-obstruction. If $G$ is a minimal $\mathcal{P}$-obstruction such
that $H \le G$, then $G \cong H$.
\end{remark}

A poset $(M, \le)$ is called a {\em well-quasi-ordering} (WQO) if any infinite
sequence of elements $\{ a_i \}_{i \in \mathbb{N}}$ from $M$ contains an
increasing pair, that is to say, a pair $a_i\le a_j$ such that $i<j$.
Equivalently, $(M, \le)$ is a WQO if and only if $M$ contains neither an
infinite decreasing chain nor an infinite antichain.

Let $\mathcal{G}$ be a graph class ordered by the induced subgraph relation, and
let $\mathcal P$ be a hereditary property on $\mathcal{G}$. By \Cref{rem:hered},
the family of minimal $\mathcal P$-obstructions is an antichain. Moreover, any
antichain in $(\mathcal{G}, \le)$ is the family of minimal $\mathcal
Q$-obstructions for a hereditary property $\mathcal Q$. Then, since graphs
ordered by the induced subgraph relation do not have infinite decreasing chains,
$\mathcal{G}$ is WQO by the induced subgraph relation if and only if it contains
no infinite antichain, or equivalently, if every hereditary property on
$\mathcal{G}$ has only finitely many minimal obstructions. Peter Damaschke
\cite{damaschkeJGT14} used the following theorem to prove that cographs and
$P_4$-reducible graphs are WQO under the induced subgraph relation.

\begin{theorem}[\cite{damaschkeJGT14}]
\label{thm:damasWQO}
Let $\mathcal{G}$ be a family of graphs, and let $\Sigma$ and $\Pi$ be sets of
unary and binary graph operations, respectively. Define partial orderings on
$\Sigma$ and $\Pi$ as follows:
\[
  \sigma\preceq \sigma' \textnormal{ if and only if } \sigma(G) \le \sigma'(G)
  \textnormal{ for all graphs } G.
\]
\[
  \pi \preceq \pi' \textnormal{ if and only if } \pi (G,H) \le \pi' (G,H)
  \textnormal{ for all graphs } G,H.
\]
Suppose that the following assertions are satisfied:
\begin{enumerate}
	\item $\mathcal{G}$ is WQO by the induced subgraph relation.

	\item Any $\sigma \in \Sigma$ is monotonous (that is, $H \le G$ implies
    $\sigma(H) \le \sigma (G)$), and extensive (that is, for any graph $G$, $G
    \le \sigma(G)$).

	\item Any $\pi \in \Pi$ is commutative, associative, and satisfies:
    \begin{enumerate}
      \item if $G \le G'$ and $H \le H'$, then $\pi(G,H) \le \pi(G',H')$, and

      \item $G,H \le \pi(G,H)$.
    \end{enumerate}

	\item $(\Sigma, \preceq)$ and $(\Pi, \preceq)$ are WQO.
\end{enumerate}
Then, the class $\Gamma(\mathcal{G}, \Sigma, \Pi)$ of all graphs obtained by
start graphs from $\mathcal{G}$ using operations from $\Sigma$ and $\Pi$, is WQO
under the induced subgraph relation.
\end{theorem}

In the following sections we provide new characterizations for both,
$P_4$-sparse and $P_4$-extendible graphs, in order to show that
\Cref{thm:damasWQO} can be used to prove that such graph families (which are
$P_4$-reducible superclasses), are WQO under the induced subgraph relation.


\subsection{Hereditary properties on $P_4$-sparse graphs}

Jamison and Olariu \cite{jamisonDAM35} gave a constructive characterization for
$P_4$-sparse graphs starting with trivial graphs and using three binary
operations. Nevertheless, the third operation used in such a characterization is
not commutative, so it does not satisfy the hypotheses of \Cref{thm:damasWQO},
and we cannot use that characterization with Damaschke's theorem to conclude
that $P_4$-sparse graphs are WQO. But, not everything is lost. Next, we
establish a different constructive characterization for $P_4$-sparse graphs that
is more appropriate for said purpose. Our characterization starts with trivial
graphs and headless spiders, and it involves two binary operations as well as
two infinite families of unary graph operations. We start with the following
straightforward observation.

\begin{remark}
\label{rem:sparsePreserv}
Disjoint union and join operations preserve $P_4$-sparse graphs.
\end{remark}

Let $H$ be a graph, and let $j$ be an integer, $j\ge 2$. The graph $\sigma_j(H)$
is the thin spider $G=(S, K, R)$ such that $|S|=|K|=j$ and $G[R]=H$.
Analogously, the graph $\tau_j(G)$ is the thick spider $G=(S, K, R)$ such that
$|S|=|K|=j$ and $G[R]=H$. Notice that $\sigma_2(H)=\tau_2(H)$ for any graph $H$.
The following observation follows directly from the definition of $P_4$-sparse
graphs.

\begin{remark}
\label{rem:spiderGen}
Let $j$ be an integer, $j\ge 2$. The graphs $\sigma_j(H)$ and $\tau_j(H)$ are
$P_4$-sparse graphs if and only if $H$ is a $P_4$-sparse graph. In addition, any
headless spider is a $P_4$-sparse graph.
\end{remark}

Let $\Pi$ be the set of binary operations whose only elements are the disjoint
union and join graph operations, and let $\Sigma = \{\sigma_j\}_{j\ge2} \cup
\{\tau_j\}_{j\ge3}$. Let us define the following partial order on $\Sigma$:
\[
  \sigma \preceq \sigma' \textnormal{ if and only if } \sigma(H) \le \sigma'(H)
  \textnormal{ for all graphs } H.
\]

It is straightforward to show that $\sigma_2 \preceq \sigma_3 \preceq \sigma_4
\preceq \cdots$ and $\sigma_2 \preceq \tau_3 \preceq \tau_4 \preceq \tau_5
\preceq \cdots$, so it follows trivially that $\Sigma$ is WQO by $\preceq$.
Analogously, it is easy to show that the family of graphs $\mathcal{G}$ whose
only elements are the trivial graph and all headless spiders is WQO under the
induced subgraph relation. Now we give our characterization of $P_4$-sparse
graphs.

\begin{theorem}
Let $G$ be a graph. The following statements are equivalent.
\begin{enumerate}
	\item $G$ is a $P_4$-sparse graph.

	\item $G$ is obtained from trivial graphs by a finite sequence of
    $\Sigma$- and $\Pi$-operations.
\end{enumerate}
\end{theorem}

\begin{proof}
	We have from \Cref{thm:connCharSparse,rem:sparsePreserv,rem:spiderGen} that 2
	implies 1. The converse implication can be easily proved proceeding by
	induction on the order of $G$ and using \Cref{thm:connCharSparse}.
\end{proof}

The theorem above shows that $\Gamma(\mathcal{G}, \Sigma, \Pi)$ is precisely the
class of $P_4$-sparse graphs. As we pointed before $(\mathcal{G}, \le)$ and
$(\Sigma,\le)$ are WQO and, since $\Pi$ is a finite set, $(\Pi,\le)$ is too, so
the following corollary is a simple application of \Cref{thm:damasWQO}.

\begin{corollary}
The class of $P_4$-sparse graphs is WQO under the induced subgraph relation.
Equivalently, any hereditary property on $P_4$-sparse graphs admits a finite
forbidden induced subgraph characterization.
\end{corollary}


\subsection{Hereditary properties on $P_4$-extendible graphs}

A constructive characterization for $P_4$-extendible graphs starting with
trivial graphs and using 4 binary operations was given in \cite{jamisonDAM34}.
As well as the constructive characterization for $P_4$-sparse graphs given in
\cite{jamisonDAM35}, this characterization for $P_4$-extendible graphs does not
fit the hypotheses of \Cref{thm:damasWQO}, so we are unable to conclude that
$P_4$-extendible graphs are WQO in this way. With this purpose in mind, we
establish a new constructive characterization for $P_4$-extendible graphs which
starts from a set of nine basic graphs and involves two binary operations as
well as five unary operations.

Let $\mathcal{G}$ be the set of graphs whose elements are the trivial graph
$K_1$ and the eight extension graphs, that is, $\mathcal{G} = \{K_1, P_4, C_5,
P_5, \overline{P_5}, P, \overline{P}, F, \overline{F}\}$. For each separable
extension graph $S$ (see \Cref{fig:extSets}) and any graph $G$, we define the
graph $\sigma_S(G)$ as the graph with vertex set $V_S \cup V_G$ and edge set
\[
  E_S \cup E_G \cup \{ xy \mid x \textnormal{ is a midpoint of $S$ and } y \in
  V_G \}.
\]
For each separable extension graph $S$, the unary operation $\sigma_S$ is its
associated {\em separable extension operation}. Let $\Sigma$ be the set of the
five separable extension operations $\sigma_S$, and let $\Pi$ be the set of
binary operations whose only elements are the disjoint union and join
operations.

\begin{remark}[\cite{jamisonDAM34}]
\label{rem:pComp}
Let $G$ be a graph whose vertex set admits a partition into two nonempty
disjoint sets $V'$ and $V''$ such that no $P_4$ in $G$ contains vertices from
both $V'$ and $V''$. Then $G$ is $P_4$-extendible if and only if the subgraphs
of $G$ induced by $V'$ and $V''$ are.
\end{remark}

Observe that, from the remark above, $P_4$-extendible graphs are clearly closed
under join and disjoint union operations. Now we use such remark for proving
that separable extension operations also preserve $P_4$-extendible graphs.

\begin{lemma}
\label{lem:sigma}
The class of $P_4$-extendible graphs is closed under separable extension
operations, that is to say, for any $P_4$-extendible graph $G$ and any separable
extension graph $S$, $\sigma_S(G)$ is a $P_4$-extendible graph.
\end{lemma}

\begin{proof}
	By definition of $\sigma_S$, the vertex set of $\sigma_S(G)$ is partitioned
	into $V_S$ and $V_G$, and by hypothesis the graphs induced by these sets are
	$P_4$-extendible. Now, by \Cref{rem:pComp} we only need to prove that no
	$P_4$ has vertices in both $V_S$ and $V_G$. Assume the contrary to obtain a
	contradiction.

	Let $M$ be the set of midpoints of $S$, and let $W$ be a vertex set inducing a
	$P_4$ such that $W \cap V_S \ne \varnothing \ne W \cap V_G$. It is an easy
	observation that, since $W$ induces a $P_4$, $|W \cap V_G|=1$, $|W \cap
	V_s|=3$, and $|W \cap M|\le 2$. So we have only two possible cases, either $|W
	\cap M|=1$ or $|W \cap M|=2$. Let $u$ be the only vertex in $W \cap V_G$.

	First, assume that $W$ has only one endpoint $x$ of $S$ and that $y$ and $z$
	are both midpoints of $S$. Let $E$ be the edge set of $\sigma_S(G)$. By
	definition of $\sigma_S$ we have that $uy,uz\in E$ and $ux \notin E$.
	Moreover, since $W$ induces a $P_4$, we have that $yz \notin E$ and $x$ is
	adjacent to exactly one of $y$ and $z$. But this is impossible, because the
	only separable extension graph with two nonadjacent midpoints is $P$, but no
	endpoint of $P$ distinguishes between its nonadjacent midpoints.

	Otherwise, $W$ has two endpoints, $y$ and $z$, and one endpoint, $x$, of $S$.
	By definition of $\sigma_S$ we have that $ux \in E$ and $uy,uz \notin E$.
	Moreover, since $W$ induces a $P_4$, we have that $yz \in E$ and $x$ is
	adjacent to exactly one of $y$ and $z$. Here we have a contradiction, because
	the only separable extension graph with two adjacent endpoints is $\overline
	P$, but no midpoint of $P$ distinguishes between its adjacent endpoints.
\end{proof}

\begin{theorem}
Let $G$ be a graph. The following statements are equivalent.
\begin{enumerate}
	\item $G$ is a $P_4$-extendible graph.

	\item $G$ is obtained from $\mathcal{G}$ by a finite sequence of $\Sigma$- and
    $\Pi$-operations.
\end{enumerate}
\end{theorem}

\begin{proof}
	The fact that 2 implies 1 follows easily from \Cref{lem:sigma}, the
	observation after \Cref{rem:pComp}, and since $\mathcal{G}$ is a subset of
	$P_4$-extendible graphs. For the converse implication we proceed by induction
	on the order of $G$. From \Cref{thm:connChar} we have that, if $G$ is not
	trivial, one of the following cases is satisfied:
	\begin{enumerate}
		\item $G$ is disconnected.

		\item $\overline{G}$ is disconnected.

		\item $G$ is an extension graph.

		\item there is a unique separable extension graph $S$ such that $G =
      \sigma_S(H)$ for some graph $H$.
	\end{enumerate}

	In the first (second) case, $G$ is the disjoint union (join) of two
	$P_4$-extendible graphs $G_1$ and $G_2$, which by induction hypothesis can be
	constructed from $\mathcal{G}$ by a finite sequence of $\Sigma$- and
	$\Pi$-operations, so the result follows in this case. The remaining cases are
  immediate.
\end{proof}

The theorem above shows that $\Gamma(\mathcal{G}, \Sigma, \Pi)$ is precisely the
class of $P_4$-extendible graphs. In addition, considering that $\mathcal{G},
\Sigma$ and $\Pi$ are finite sets, it is easy to justify the following
consequence of \Cref{thm:damasWQO}.

\begin{corollary}
The class of $P_4$-extendible graphs is WQO under the induced subgraph relation.
Equivalently, any hereditary property on $P_4$-extendible graphs admits a finite
forbidden subgraph characterization.
\end{corollary}

The rest of the paper is devoted to the characterizations by forbidden induced
subgraphs of properties associated with polarity in $P_4$-sparse and
$P_4$-extendible graphs. We start with the characterization of minimal unipolar
obstructions on the mentioned graph classes.


\section{Minimal unipolar obstructions}
\label{sec:MinUnipObs}

In this section we provide complete lists of minimal unipolar obstructions which
are $P_4$-sparse or $P_4$-extendible graphs. With that purpose in mind we
introduce some minimal unipolar obstructions that do not necessarily belong to
the mentioned graph classes.

A {\em hole} is a cycle of length at least 5. An {\em antihole} is the
complement of a hole. Holes and antiholes are said to be {\em even} or {\em odd}
accordingly to their order.

\begin{proposition}
\label{prop:someMUO}
The graphs depicted in \Cref{fig:muo} are minimal unipolar obstructions.
\end{proposition}

\begin{figure}[hb!]
\centering
\begin{tikzpicture}[scale=0.85]

\begin{scope}[scale=1, xshift=-0.5cm, yshift=-1cm]
\node (1) [vertex] at (0,0){};
\node (2) [vertex] at (0,1){};
\node (3) [vertex] at (0,2){};
\node (5) [vertex] at (1,0){};
\node (6) [vertex] at (1,1){};
\node (7) [vertex] at (1,2){};

\node (x) [rectangle] at (0.5,-0.5){$2P_3$};
\end{scope}

\foreach \i in {1,2,5,6}
	\draw let \n1={int(\i+1)} in [edge] (\i) to node [above] {} (\n1);

\begin{scope}[scale=1, xshift=4cm, yshift=0cm]
\foreach \i in {0,...,3}
	\node (\i) [vertex] at ({45+(\i*90)}:1){};

\node (c) [vertex] at (0,0){};

\node (x) [rectangle] at (0,-1.5){$K_{2,3}$};
\end{scope}

\foreach \i in {0,...,3}
	\draw let \n1={int(mod(\i+1,4))} in [edge] (\i) to node [above] {} (\n1);

\draw [edge] (c) to node [above] {} (0);
\draw [edge] (c) to node [above] {} (2);

\begin{scope}[scale=1, xshift=8cm, yshift=0cm]
\foreach \i in {0,...,6}
	\node (\i) [vertex] at ({90+(\i*(360/7))}:1){};

\node (x) [rectangle] at (0,-1.5){odd antiholes};
\end{scope}

\foreach \i in {0,...,6}
	\draw let \n1={int(mod(\i+2,7))} in [edge] (\i) to node [above] {} (\n1);
\foreach \i in {0,...,6}
	\draw let \n1={int(mod(\i+3,7))} in [edge] (\i) to node [above] {} (\n1);

\end{tikzpicture}
\caption{Some minimal unipolar obstructions.}
\label{fig:muo}
\end{figure}
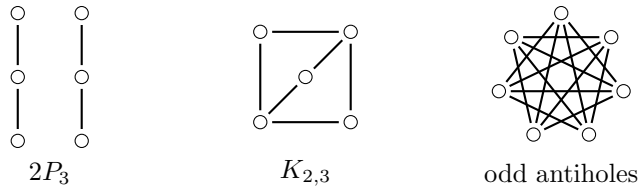

\begin{proof}
  To prove that these graphs are not unipolar, it is enough to observe that for
  any clique $K$, $G-K$ has an induced $P_3$, so $G-K$ is not a cluster. It is
  also easy to verify that any vertex-deleted subgraph of these graphs is a
  unipolar graph, so the result follows.
\end{proof}

The following two lemmas completely characterize minimal unipolar obstructions
$G$ (on general graphs) such that either $G$ or $\overline{G}$ is disconnected.
We use such characterizations as the base for providing complete lists of
minimal unipolar obstructions for cographs, $P_4$-sparse graphs and
$P_4$-extendible graphs.

\begin{lemma}
\label{lem:discMUO}
If $G$ is a graph, then $G$ is a disconnected minimal unipolar obstruction if
and only if $G \cong 2P_3$.
\end{lemma}

\begin{proof}
  Let $G$ be a disconnected minimal unipolar obstruction. By the minimality of
  $G$, any of its components is a unipolar graph. In consequence, $G$ has at
  least two components which are not complete graphs, otherwise $G$ would be
  unipolar. Then, $G$ has $2P_3$ as an induced subgraph, so $G \cong 2P_3$.
  The converse implication follows from \Cref{prop:someMUO}.
\end{proof}

\begin{lemma}
\label{lem:COdiscMUO}
Let $G$ be a graph.   If $\overline{G}$ is disconnected, then $G$ is a minimal
unipolar obstruction if and only if $G \cong K_{2,3}$.
\end{lemma}

\begin{proof}
  First, suppose that $G$ is a minimal unipolar obstruction. Notice that
  $\overline{G}$ is not a bipartite graph, or $G$ would admit a partition into
  two cliques, so it would be a unipolar graph, which is impossible. Hence
  $\overline{G}$ contains an odd cycle as an induced subgraph. Moreover, since
  odd antiholes are minimal unipolar obstructions and $\overline{G}$ is
  disconnected, $\overline{G}$ does not contain odd cycles of length greater
  than three as induced subgraphs. Thus, $\overline{G}$ contains a triangle. In
  addition, since minimal unipolar obstructions do not have universal vertices,
  $\overline{G}$ does not have isolated vertices, and any component of
  $\overline{G}$ has order at least two. Therefore, since $\overline{G}$ has at
  least two connected components, it contains $K_2 + K_3$ as an induced
  subgraph, so $G \cong K_{2,3}$. The converse implication follows from
  \Cref{prop:someMUO}.
\end{proof}

Since the complement of any nontrivial connected cograph is a disconnected
cograph we have the following direct consequence of
\Cref{lem:discMUO,lem:COdiscMUO}.

\begin{corollary}
\label{cor:cogMUO}
If $G$ is a cograph, then $G$ is a minimal unipolar obstruction if and only if
$G \cong 2P_3$ or $G \cong K_{2,3}$.
\end{corollary}

Now, we use the characterization of $P_4$-sparse graphs given in
\Cref{thm:connCharSparse} to give the explicit list of $P_4$-sparse minimal
unipolar obstructions.

\begin{lemma}
\label{lem:spNoMUO}
If $G = (S, K, R)$ is a spider, then $G$ is a unipolar graph if and only if $R =
\varnothing$ or $G[R]$ is unipolar.
\end{lemma}

\begin{proof}
  Since unipolarity is a hereditary property, we have that $G[R]$ is unipolar
  whenever $G$ is. Conversely, for any unipolar partition $(A, B)$ of $G[R]$,
  $(K \cup A, S \cup B)$ is a unipolar partition of $G$.
\end{proof}

\begin{corollary}
\label{cor:spMUO}
If $G$ is a $P_4$-sparse graph, then $G$ is a minimal unipolar obstruction if
and only if $G \cong 2P_3$ or $G \cong K_{2,3}$. In consequence, any
$P_4$-sparse minimal unipolar obstruction is a cograph.
\end{corollary}

\begin{proof}
  The first statement follows from
  \Cref{lem:discMUO,lem:COdiscMUO,lem:spNoMUO}, since we have by
  \Cref{thm:connCharSparse} that any connected $P_4$-sparse graph with connected
  complement is a spider. The second statement follows directly from
  \Cref{cor:cogMUO}.
\end{proof}

We end this section by proving a result analogous to \Cref{cor:spMUO} for
$P_4$-extendible graphs. Notice that, for any extension graph but $P$, its
midpoints are a clique while its endpoints induce a cluster (see
\Cref{fig:extSets}). Then, the proof of the following proposition is exactly the
same as the proof of \Cref{lem:spNoMUO}. The case of $P$-spiders is covered in
\Cref{lem:extMUO2}.

\begin{lemma}
\label{lem:extMUO1}
Let $H \in \{P_4, \overline{P}, F, \overline{F}\}$.   If $G = (S, K, R)$ is an
$H$-spider, then $G$ is a unipolar graph if and only if $R = \varnothing$ or
$G[R]$ is unipolar.
\end{lemma}

\begin{lemma}
\label{lem:extMUO2}
If $G = (S, K, R)$ is a $P$-spider, then $G$ is a unipolar graph if and only if
either $R$ is an empty set or a clique. In consequence, if $G$ is a $P$-spider,
then it is not a minimal unipolar obstruction.
\end{lemma}

\begin{proof}
  Let $w$ be the only vertex of $G[S \cup K]$ of degree 2 which is not adjacent
  to a vertex of degree three, and let $u$ and $v$ be its neighbors. If $R$ has
  two nonadjacent vertices $x$ and $y$, then $G[\{u, v, w, x, y\}]$ is
  isomorphic to $K_{2,3}$. Therefore, if $G$ is a unipolar graph, then $R =
  \varnothing$ or $R$ is a clique. Conversely, if $R$ is a clique and $z$ is the
  only vertex of $G[S \cup K]$ of degree three, then $(R \cup \{z, u\}, (S \cup
  K) \setminus \{z, u\})$ is a unipolar partition of $G$. Hence, if $G$ is not a
  unipolar graph, $R$ contains two nonadjacent vertices and $G$ properly
  contains $K_{2,3}$, so $G$ is not a minimal unipolar obstruction.
\end{proof}

\begin{corollary}
Let $G$ be a $P_4$-extendible graph.  Then, $G$ is a minimal unipolar
obstruction if and only if $G\in \{ 2P_3, K_{2,3}, C_5\}$.
\end{corollary}

\begin{proof}
  We have from \Cref{thm:connChar} that any connected $P_4$-extendible graph
  with connected complement is either an extension graph, or an $S$-spider for
  some separable extension graph $S$. It is easy to verify that the only
  extension graph which is a minimal unipolar obstruction is $C_5$, so the
  result follows from \Cref{lem:discMUO,lem:COdiscMUO,lem:extMUO1,lem:extMUO2}.
\end{proof}

Now, we study $(s,1)$-polarity, monopolarity and polarity in $P_4$-sparse and
$P_4$-extendible graphs. We start by giving complete lists of disconnected
minimal $(s,1)$-polar obstructions for general graphs in \Cref{sec:discM(s1)PO},
followed by some technical results in \Cref{sec:connectedM(s1)PO} directed to
provide complete lists of $P_4$-sparse and $P_4$-extendible minimal
$(s,1)$- $(\infty, 1)$-, and $(\infty, \infty)$-polar obstructions in
\Cref{sec:main}.


\section{Disconnected minimal $(s,1)$-polar obstructions}
\label{sec:discM(s1)PO}

The following five lemmas completely characterize disconnected minimal
$(s,1)$-polar obstructions for general graphs. They are simple generalizations
of Lemmas 2 to 5 from \cite{contrerasDAM281}, so we will only sketch the proofs.

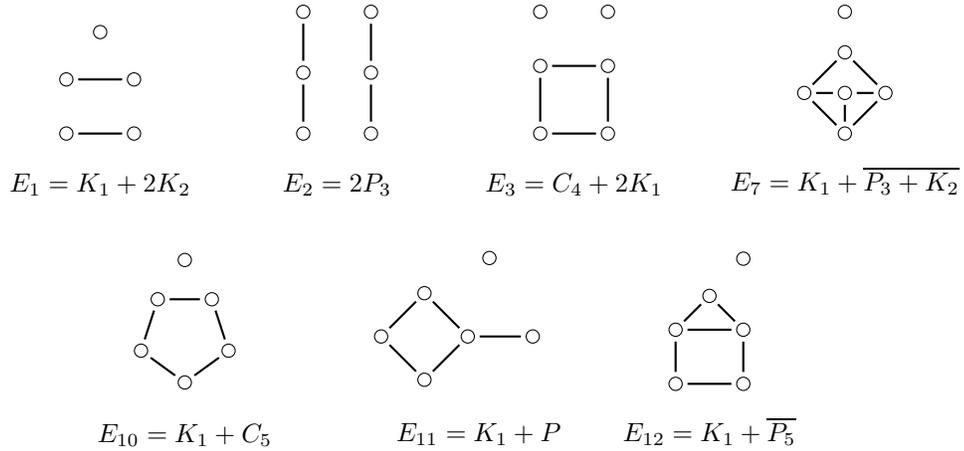
\begin{figure}[ht!]
\centering
\begin{tikzpicture}
\begin{scope}[scale=0.9]

\begin{scope}[scale=1]
\node [vertex] (1) at (0,0)[]{};
\node [vertex] (2) at (1,0)[]{};
\node [vertex] (3) at (0,0.8)[]{};
\node [vertex] (4) at (1,0.8)[]{};
\node [vertex] (5) at (0.5,1.5)[]{};

\foreach \from/\to in {1/2,3/4}
  \draw [edge] (\from) to (\to);

\node (1) at (0.5,-0.75){$E_1 = K_1+2K_2$};
\end{scope}

\begin{scope}[xshift=3.5cm, scale=1]
\node [vertex] (1) at (0,0)[]{};
\node [vertex] (2) at (0,0.9)[]{};
\node [vertex] (3) at (0,1.8)[]{};
\node [vertex] (4) at (1,0)[]{};
\node [vertex] (5) at (1,0.9)[]{};
\node [vertex] (6) at (1,1.8)[]{};

\foreach \from/\to in {1/2,2/3,4/5,5/6}
  \draw [edge] (\from) to (\to);

\node (1) at (0.5,-0.75){$E_2 = 2P_3$};
\end{scope}

\begin{scope}[xshift=7cm, scale=1]
\node [vertex] (1) at (0,0)[]{};
\node [vertex] (2) at (1,0)[]{};
\node [vertex] (3) at (0,1)[]{};
\node [vertex] (4) at (1,1)[]{};
\node [vertex] (5) at (0,1.8)[]{};
\node [vertex] (6) at (1,1.8)[]{};

\foreach \from/\to in {1/2,1/3,3/4,2/4}
  \draw [edge] (\from) to (\to);

\node (1) at (0.5,-0.75){$E_3 = C_4+2K_1$};
\end{scope}

\begin{scope}[xshift=11cm, scale=1]
\node [vertex] (1) at (0.5,0)[]{};
\node [vertex] (2) at (-0.1,0.6)[]{};
\node [vertex] (3) at (0.5,1.2)[]{};
\node [vertex] (4) at (1.1,0.6)[]{};
\node [vertex] (5) at (0.5,0.6)[]{};
\node [vertex] (6) at (0.5,1.8)[]{};

\foreach \from/\to in {1/2,2/3,3/4,4/1,1/5,2/5,4/5}
  \draw [edge] (\from) to (\to);

\node (1) at (0.5,-0.7){$E_7 = K_1+\overline{P_3+K_2}$};
\end{scope}

\begin{scope}[xshift=1.75cm, yshift=-3cm, scale=0.85]
\node [vertex] (1) at (126:0.8)[]{};
\node [vertex] (2) at (54:0.8)[]{};
\node [vertex] (3) at (342:0.8)[]{};
\node [vertex] (4) at (270:0.8)[]{};
\node [vertex] (5) at (198:0.8)[]{};
\node [vertex] (6) at (90:1.33)[]{};

\foreach \from/\to in {1/2,2/3,3/4,4/5,5/1}
  \draw [edge] (\from) to (\to);

\node (1) at (0,-1.72){$E_{10} = K_1+C_5$};
\end{scope}

\begin{scope}[xshift=6.25cm, yshift=-3cm, scale=0.8]
\node [vertex] (1) at (-2,0)[]{};
\node [vertex] (2) at (-1.2,0.8)[]{};
\node [vertex] (3) at (-0.4,0)[]{};
\node [vertex] (4) at (-1.2,-0.8)[]{};
\node [vertex] (5) at (0.8,0)[]{};
\node [vertex] (6) at (0,1.45)[]{};

\foreach \from/\to in {1/2,2/3,3/4,4/1,3/5}
  \draw [edge] (\from) to (\to);

\node (1) at (-0.2,-1.8){$E_{11} = K_1+P$};
\end{scope}

\begin{scope}[xshift=9cm, yshift=-2.7cm, scale=1]
\node [vertex] (1) at (0,-1)[]{};
\node [vertex] (2) at (0,-0.2)[]{};
\node [vertex] (3) at (1,-0.2)[]{};
\node [vertex] (4) at (1,-1)[]{};
\node [vertex] (5) at (0.5,0.3)[]{};
\node [vertex] (6) at (1,0.85)[]{};

\foreach \from/\to in {1/2,2/3,3/4,4/1,2/5,5/3}
  \draw [edge] (\from) to (\to);

\node (1) at (0.5,-1.7){$E_{12} = K_1+\overline{P_5}$};
\end{scope}

\end{scope}
\end{tikzpicture}
\caption{Some minimal $(\infty,1)$-polar obstructions.}
\label{fig:essentials}
\end{figure}

\begin{lemma}
\label{lem:discMonopObsmin}
The seven graphs depicted in \Cref{fig:essentials} are minimal $(s,1)$-polar
obstructions for every integer $s$, $s\ge 2$. Hence, these graphs are minimal
$(\infty, 1)$-polar obstructions.
\end{lemma}

\begin{proof}
	It is routine to verify that, for each graph $G$ in \Cref{fig:essentials}, the
	following assertions are satisfied:  For any maximal clique $K$, $G-K$
	contains an induced $\overline{P_3}$, and for any vertex $v$, $G-v$ is a
	$(2,1)$-polar graph. The result follows easily from here.
\end{proof}

In \cite{contrerasDAM281}, a proof of the following two lemmas restricted to the
family of cographs was given. A minor change in such a proof brings us the more
general results that we state here.

\begin{lemma}
\label{lem:twoComponentsA}
Let $s$ be an integer, $s\ge 2$. Every minimal $(s, 1)$-polar obstruction
different from $K_1 + 2K_2$ and $2K_1 + C_4$ has at most two connected
components.
\end{lemma}

\begin{proof}
	Let $G$ be a minimal $(s,1)$-polar obstruction different from the graphs
	depicted in \Cref{fig:essentials}. Assume for a contradiction that $G$ has
	at least three connected components. Since $s \ge 2$, we have that $G$ is
	not a split graph, so it contains $2K_2, C_4$ or $C_5$ as an induced
	subgraph. Having at least three connected components, $G$ contains some of
	$K_1 + 2K_2, 2K_1 + C_4$ or $K_1 + C_5$ as an induced subgraph. This results
	in a contradiction, because these graphs are minimal $(s,1)$-polar
	obstructions. Thus, $G$ has at most two components.
\end{proof}

\begin{lemma}[\cite{contrerasDAM281}]
\label{lem:twoComponentsB}
Let $s$ be an integer, $s\ge 2$. If a minimal $(s,1)$-polar obstruction $G$
distinct to the graphs depicted in \Cref{fig:essentials} has two connected
components and it is not $2K_{s+1}$, then $G \cong K_r + H$, where $r \in
\{1,2\}$ and $H$ is a connected graph that is not a complete $s$-partite graph.
\end{lemma}

The proof of the following lemma has the same spirit than the proof of Lemma 4
in \cite{contrerasDAM281}, but it has been rewritten for the sake of clarity.

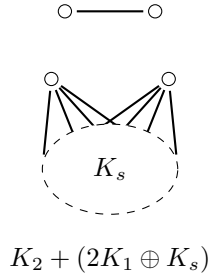
\begin{figure}[ht!]
\begin{center}
\begin{tikzpicture}

\begin{scope}[scale=0.6]
\node [vertex] (0) at (-1,3.5){};
\node [vertex] (1) at (1,3.5){};
\node [vertex] (2) at (-1.3,2){};
\node [vertex] (3) at (1.3,2){};

\node (4) at (-1.5,0){};
\node (5) at (-0.5,0){};
\node (7) at (0.5,0){};
\node (8) at (1.5,0){};

\foreach \from/\to in {0/1,2/4,2/5,2/7,2/8,3/4,3/5,3/7,3/8}
  \draw [edge] (\from) to (\to);
\end{scope}

\begin{scope}[scale=0.6]
\draw[dashed, fill=white] (0,0) ellipse (1.5cm and 1cm);
\node (10) at (0,0){$K_s$};
\node (10) at (0,-2){$K_2 + (2K_1 \oplus K_s)$};
\end{scope}

\end{tikzpicture}
\end{center}
\caption{The only minimal $(s,1)$-polar obstruction different from $2K_{s+1}$
with exactly two connected components one of them being isomorphic to $K_2$.}
\label{fig:ObsTypeK2}
\end{figure}

\begin{lemma}
\label{lem:H+K2}
Let $s$ be an integer, $s \ge 2$.   If $H$ is a connected graph such that $G
= K_2 + H$ is a minimal $(s,1)$-polar obstruction other than $2K_{s+1}$, then
$H$ is isomorphic to $2K_1 \oplus K_s$.
\end{lemma}

\begin{proof}
	It is routine to verify that $K_2 + (2K_1 \oplus K_s)$ is a minimal
	$(s,1)$-polar obstruction. From \Cref{lem:twoComponentsB} we know that $H$
	is not a complete $s$-partite graph, so $H$ contains a copy of either
	$\overline {P_3}$ or $K_{s+1}$ as an induced subgraph. Nevertheless, $H$ is
	a $\overline{P_3}$-free graph, for otherwise $G$ would contain $K_1+2K_2$ as
	a proper induced subgraph. Thus, $H$ contains a copy of $K_{s+1}$ as a
	proper induced subgraph. Let $K$ be a maximum clique in $H$, and let $v\in
	V_H - K$. As we argued above, $H$ is a $\overline{P_3}$-free graph, so $v$
	is adjacent to all but one vertex $w$ in $K$. Hence, for any $s$-subset $V'$
	of $K \cap N(v)$, the graph $H[V'\cup \{v,w\}]$ is isomorphic to $2K_1\oplus
	K_s$, so $G \cong K_2 + (2K_1 \oplus K_s)$.
\end{proof}

The following lemma is a slight generalization of Lemma 5 in
\cite{contrerasDAM281}.   Since the main ideas of the proof are very similar, we
will only explain the significant differences.

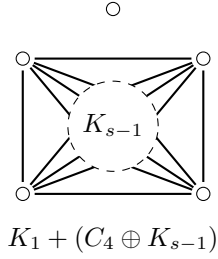
\begin{figure}[ht!]
\centering
\begin{tikzpicture}

\begin{scope}[scale=0.6]
\node (10) at (0,-2){\color{white}$K_2 + (2K_1 \oplus K_s)$};
\end{scope}

\begin{scope}[scale=0.6, yshift=-0.3cm]
\node [vertex] (0) at (0,3.6){};

\node [vertex] (0) at (2,2.5){};
\node [vertex] (1) at (2,-0.5){};
\node [vertex] (2) at (-2,2.5){};
\node [vertex] (3) at (-2,-0.5){};

\node [vertex, yshift=0.6cm] (4) at (90:0.75){};
\node [vertex, yshift=0.6cm] (5) at (45:0.75){};
\node [vertex, yshift=0.6cm] (6) at (0:0.75){};
\node [vertex, yshift=0.6cm] (7) at (315:0.75){};
\node [vertex, yshift=0.6cm] (8) at (270:0.75){};
\node [vertex, yshift=0.6cm] (9) at (225:0.75){};
\node [vertex, yshift=0.6cm] (10) at (180:0.75){};
\node [vertex, yshift=0.6cm] (11) at (135:0.75){};

\foreach \from/\to in {0/1,1/3,3/2,2/0,0/4,0/5,0/6,1/6,1/7,1/8,3/8,3/9,3/10,
2/10,2/11,2/4}
  \draw [edge] (\from) to (\to);
\end{scope}

\begin{scope}[scale=0.6, yshift=0.7cm]
\draw[dashed, fill=white] (0,0) ellipse (1cm and 1cm);
\node (10) at (0,0){$K_{s-1}$};
\node (10) at (0,-2.5){$K_1 + (C_4 \oplus K_{s-1})$};
\end{scope}

\end{tikzpicture}
\caption{The only minimal $(s,1)$-polar obstruction different from those on
\Cref{fig:essentials} with exactly two connected components one of them being
isomorphic to $K_1$.}
\label{fig:ObsTypeK1}
\end{figure}

\begin{lemma}
\label{lem:H+K1}
Let $s$ be an integer, $s\ge 2$.   If $H$ is a connected graph such that $G =
K_1 + H$ is a minimal $(s,1)$-polar obstruction isomorphic to none of the graphs
depicted in \Cref{fig:essentials}, then $G$ is isomorphic to $K_1 + (C_4\oplus
K_{s-1})$.
\end{lemma}

\begin{proof}
	The graph $H$ cannot be a split graph, so it contains an induced copy of
	either $2K_2, C_4$ or $C_5$. Nevertheless, by the minimality of $G$ and since
	$G \not \cong K_1+C_5$ we know that $H$ is $\{2K_2,C_5\}$-free, so it contains
	an induced cycle on four vertices, $C=(c_1,c_2,c_3,c_4)$. Let $v$ be a vertex
	in $H-V_C$, which must exist since $H$ is not a complete bipartite graph.
	Observe that, since $G$ contains no graph depicted in \Cref{fig:essentials} as
	an induced subgraph, $v$ only could be adjacent to either two nonadjacent
	vertices of $C$ or to every vertex of $C$.

	Let $V_1, V_2$ and $V_3$ be the subsets of vertices of $H$ that are not in $C$
	and that are adjacent to $c_1$ and $c_3$, to $c_2$ and $c_4$, and to $c_i$ for
	every $i \in \{1,2,3,4\}$, respectively. Notice that since $H$ contains no
	induced $\overline{K_2+P_3}$, $V_1$ an $V_2$ are both independent sets, and
	$V_3$ is completely adjacent to $V_1 \cup V_2$. In addition, since $H$ is
	$P$-free, we have that $V_1$ and $V_2$ are completely adjacent. From here is
	straightforward to notice that $H-V_3$ is a complete bipartite graph.

	Hence, $H$ is the join of the complete bipartite graph $H-V_3$ with $H[V_3]$,
	which implies that $H[V_3]$ is not a complete $(s-2)$-partite graph. One more
	time, since $\overline{K_2+P_3}$ is not an induced subgraph of $H$ we have
	that $H$ is a $\overline{P_3}$-free graph, so $H[V_3]$ is too. Therefore,
	$H[V_3]$ contains a copy of $K_{s-1}$ as an induced subgraph, so the result
	follows.
\end{proof}

So far, we have characterized all disconnected minimal $(s,1)$-polar
obstructions, which are a constant number for any choice of $s$. We summarize
this result as follows.

\begin{theorem}
\label{thm:allDisc(s1)obsmin}
Let $s$ be an integer, $s\ge 2$, and let $G$ be a disconnected minimal
$(s,1)$-polar obstruction. Then $G$ satisfies one of the following assertions:
\begin{enumerate}
	\item $G$ is isomorphic to one of the graphs depicted in
    \Cref{fig:essentials}.

	\item $G \cong 2K_{s+1}$.

	\item $G \cong K_2 + (2K_1\oplus K_s)$.

	\item $G \cong K_1 + (C_4\oplus K_{s-1})$.
\end{enumerate}
\end{theorem}

For any nontrivial cograph $G$, either $G$ or its complement is disconnected
\cite{corneilDAM3}, so the complement of any nontrivial connected cograph is
disconnected. This fact was used in \cite{contrerasDAM281} to give a recursive
characterization of all cograph minimal $(s,1)$-polar obstructions. After giving
a complete characterization of the disconnected cograph minimal $(s,1)$-polar
obstructions, the authors provided a recursive construction for the disconnected
cograph minimal $(1,s)$-polar obstructions (which are precisely the complements
of connected cograph minimal $(s,1)$-polar obstructions).

In the following section we will present very similar results for $P_4$-sparse
graphs and $P_4$-extendible graphs. In particular we will prove that all
$P_4$-sparse minimal $(s,1)$-polar obstructions are cographs, which turns out to
be similar in flavor to a result obtained in \cite{hannnebauerTH'10}, stating
that all $P_4$-sparse minimal obstructions for $(k,\ell)$-coloring are cographs.


\section{Connected minimal $(s,1)$-polar obstructions}
\label{sec:connectedM(s1)PO}

\Cref{thm:allDisc(s1)obsmin} characterizes disconnected minimal $(s,1)$-polar
obstructions for general graphs.   Thus, to completely characterize minimal
$(s,1)$-polar obstructions for a given class of graphs it suffices to
characterize connected minimal $(s,1)$-polar obstructions.   To this end, in
order to follow the strategy described in the previous paragraphs for
$P_4$-sparse and $P_4$-extendible graphs, we notice that the following lemma,
which was stated and proved in \cite{contrerasDAM281} for the special case of
cographs, is also valid for general graphs.

\begin{lemma}
\label{lem:basicConnected}
Let $t$ be an integer, $t \ge 2$, and for each $i\in \{1, \dots, t\}$, let $G_i$
be a minimal $(1, k_i)$-polar obstruction that is a $(1, k_i+1)$-polar graph.
Then, for $k = t-1 + \sum_{i=1}^t k_i$, the graph $G = G_1 + \dots + G_t$ is a
minimal $(1, k)$-polar obstruction that is a $(1, k+1)$-polar graph.
\end{lemma}

In the following sections we show that the converse of \Cref{lem:basicConnected}
holds for $P_4$-sparse and $P_4$-extendible graphs, that is to say, that any
disconnected minimal $(1,k)$-polar obstruction on such classes is the disjoint
union of minimal $(1,k_i)$-polar obstructions for some integers $k_i < k$.


\subsection{Connected $P_4$-sparse minimal $(s,1)$-polar obstructions}

The induced path on three vertices is a minimal $(0, k)$-polar obstruction for
any integer $k \ge 2$ so, if a graph $G$ contains $P_3$ as a proper induced
subgraph, then $G$ is not a minimal $(0, k)$-polar obstruction. Similarly, if
$G$ contains $\overline{P_3}$ as a proper induced subgraph, then $G$ is not a
minimal $(s, 0)$-polar obstruction. From here, the following observation follows
easily.

\begin{remark}
\label{rem:spiderSplit}
Let $G$ be a spider. If $G$ is a headless spider or the head of $G$ induces a
split graph, then $G$ is a split graph that has both, $P_3$ and its complement,
as proper induced subgraphs. Hence, $G$ is not a minimal $(s,k)$-polar
obstruction for any choice of $s$ and $k$.
\end{remark}

The following two propositions provide the basis for showing that any connected
$P_4$-sparse minimal $(k,1)$-polar obstruction has a disconnected complement.

\begin{proposition}
\label{prop:Sp(1k)}
Let $k$ be a positive integer, and let $G = (S, K, R)$ be a spider with
nonempty head. Then, $G$ is not a minimal $(1,k)$-polar obstruction.
\end{proposition}

\begin{proof}
	Suppose for a contradiction that $G$ is a minimal $(1,k)$-polar obstruction,
	and let $\sigma \in S$ be a leg of $G$. Let $(A,B)$ be a $(1,k)$-polar
	partition of $G-\sigma$. Notice that $|K \cap A| \le 1$ because $K$ is a
	clique and $A$ is an independent set. Therefore, since $K$ has at least two
	vertices, $K \cap B \ne \varnothing$. Moreover, since $B$ induces a cluster,
	$R$ is completely adjacent to $K$, and $K \cap B \ne \varnothing$, $R \cap B$
	is a clique. Also notice that either $K \cap A = \varnothing$ or $R \cap A =
	\varnothing$.

	Now, if $K \cap A \ne \varnothing$, then $R$ is a clique, $G$ is a split
	graph, and therefore $G$ is a $(1,k)$-polar graph, which is impossible.
	Otherwise, if $K \subseteq B$, then $(A \cup \{\sigma\}, B)$ is a
	$(1,k)$-polar partition of $G$, a contradiction.
\end{proof}

Since the complement of a spider is also a spider, and any minimal
$(\infty,1)$-polar obstruction is a minimal $(k,1)$-polar obstruction for some
positive integer $k$, we have the following simple consequences of the previous
proposition.

\begin{corollary}
\label{cor:spiderNO(s1)obsmin}
Let $k$ be a positive integer.   If $G$ is a spider, then $G$ is neither a
minimal $(k,1)$-polar obstruction nor a minimal $(\infty,1)$-polar obstruction.
\end{corollary}

\begin{corollary}
\label{cor:spiderNOobsmin}
Let $k$ be a positive integer.   If $G$ is a $P_4$-sparse minimal $(k,1)$-polar
obstruction, then $G$ or its complement is disconnected.
\end{corollary}

	\begin{proof}
	Since $G$ is a $P_4$-sparse graph, if $G$ and $\overline G$ are connected, we
	have from \Cref{thm:connCharSparse} that $G$ is a spider, but that is
	impossible by \Cref{cor:spiderNO(s1)obsmin}. Therefore, either $G$ or its
	complement is disconnected.
	\end{proof}

The next two results, together with \Cref{lem:basicConnected}, provide us with a
complete structural characterization for disconnected $P_4$-sparse minimal
$(1,k)$-polar obstructions.

\begin{lemma}
\label{lem:kappa}
Let $t$ be an integer, $t\ge 2$, and for each $i\in \{1,\dots,t\}$, let $G_i$ be
a connected $P_4$-sparse minimal $(1,k_i)$-polar obstruction which is a $(1,
k_i+1)$-polar graph. If $G = G_1 + \dots + G_t$, then $G$ is a minimal
$(1,k)$-polar obstruction if and only if $k = t-1+\sum_{i=1}^t k_i$.
\end{lemma}

	\begin{proof}
	Let $k = t-1+\sum_{i=1}^t k_i$. We have from \Cref{lem:basicConnected} that
	$G$ is a minimal $(1,k)$-polar obstruction which is $(1,k+1)$-polar, so we
	just need to show that $G$ is not a minimal $(1,\kappa)$-polar obstruction for
	any $\kappa<k$.

	Let $G$ be a connected $P_4$-sparse minimal $(1, k_i)$-polar obstruction which
	is $(1, k_i+1)$-polar. By \Cref{cor:spiderNOobsmin}, $\overline G$ is a
	disconnected minimal $(k_i, 1)$-polar obstruction which is a $(k_i+1,
	1)$-polar graph. Then, it follows from \Cref{thm:allDisc(s1)obsmin} that, for
	any nonnegative integer $\kappa_i$ such that $\kappa_i < k_i$, $G$ contains a
	proper induced subgraph $G'$ that is both, a $P_4$-sparse minimal $(\kappa_i,
	1)$-polar obstruction and a $(\kappa_i+1, 1)$-polar graph. From here on, the
	proof follows as the proof of Lemma 8 in \cite{contrerasDAM281}.
	\end{proof}

\begin{lemma}
\label{lem:disc(1k)obsmin}
Let $k$ be a nonnegative integer. If $G$ is a disconnected $P_4$-sparse minimal
$(1,k)$-polar obstruction with components $G_1, \dots, G_t$, then there exist
nonnegative integers $k_1, \dots, k_t$ such that for each $i \in \{1,\dots,t\}$,
$G_i$ is a connected minimal $(1,k_i)$-polar obstruction that is a
$(1,k_i+1)$-polar graph, and $\sum_{i=1}^t k_i = k-t+1$. (Notice that $k_i<k$
for any $i \in \{1, \dots, t\}$, and $G$ is a $(1,k+1)$-polar graph.)
\end{lemma}

	\begin{proof}
	This is a generalization of Lemma 9 in \cite{contrerasDAM281}, which states
	the same result for cographs. As in the original proof, it is easy to argue
	that each component $G_i$ is a minimal $(1,k_i)$-polar obstruction that is
	$(1,k)$-polar, where $k_i$ is the maximum integer such that any proper induced
	subgraph of $G_i$ is $(1,k_i)$-polar.

	Then, by \Cref{cor:spiderNOobsmin}, $\overline{G_i}$ is a disconnected minimal
	$(k_i,1)$-polar obstruction that is $(k,1)$-polar, and we have from
	\Cref{thm:allDisc(s1)obsmin} that $\overline{G_i}$ is a $(k_i+1,1)$-polar
	graph, so $G_i$ is $(1,k_i+1)$-polar. Finally, the result follows from
	\Cref{lem:kappa}.
	\end{proof}

Analogous results to those obtained in this section for $P_4$-sparse graphs will
be given for $P_4$-extendible graphs in the next section. As the reader can
notice, the technique used to obtain the results for both classes is the same,
the differences come from the connectedness characterizations for said families.


\subsection{Connected $P_4$-extendible minimal $(s,1)$-polar obstructions}

We begin with some easily verifiable facts, stated without proof and bundled to
facilitate future references.

\begin{remark}
\label{rem:extNoObs}
Let $s,k$ be either in $\mathbb{N}$ or equal to $\infty$.
\begin{enumerate}
  \item \label{part:P4F} $P_4$ and $F$ are split graphs but they are neither
    $(0,\infty)$- nor $(\infty, 0)$-polar graphs.

	\item \label{part:C5P5P} $C_5, P_5$, and $P$ are $(1,2)$- and $(2,1)$-polar,
	 but they are neither $(1,1)$-, $(\infty,0)$- nor $(0,\infty)$-polar graphs.

  \item \label{part:ext} An extension graph $G$ is a minimal $(s,k)$-polar
    obstruction if and only if $G \cong C_5$ and $s = k = 1$.
\end{enumerate}
\end{remark}

The following proposition allows us to show that any connected $P_4$-extendible
minimal $(1, k)$-polar obstruction, other than $C_5$, has a disconnected
complement.

\begin{lemma}
\label{lem:SspiderNoObs}
Let $k$ be a nonnegative integer, and let $G$ be a separable extension graph.
If $H = (S, K, R)$ is a $G$-spider with nonempty head, then $H$ is not a
minimal $(1,k)$-polar obstruction.
\end{lemma}

\begin{proof}
	The proof is divided in three cases, depending on $G$. If $G$ is isomorphic to
	$P_4, F$ or $\overline{F}$ then the midpoints set of $G$ form a clique with at
	least two vertices, while its endpoints set is an independent set, in which
	case the proof is the same as \Cref{prop:Sp(1k)}.

	Now, assume that $G \cong P$. Let $v$ be the only vertex of $G$ of degree one,
	and let $w$ be the support vertex of $v$; notice that $w$ is a midpoint of
	$G$. Assume for a contradiction that $H$ is a minimal $(1,k)$-polar
	obstruction, and let $(A,B)$ be a $(1,k)$-polar partition of $H-v$. If $w \in
	A$, then there are two midpoints of $G$ in $B$, but in such a case $R \cap A$
	and $R \cap B$ are both empty sets, which is impossible. Then, $w \in B$ and
	$(A \cup \{v\}, B)$ is a $(1,k)$-polar partition of $H$, a contradiction.
	Hence, $H$ is not a minimal $(1,k)$-polar obstruction.

	Finally, assume that $G\cong \overline P$, and let $v$ and $w$ as in the
	previous paragraph. Suppose for a contradiction that $H$ is a minimal
	$(1,k)$-polar obstruction, and let $(A,B)$ be a $(1,k)$-polar partition of
	$H-v$. If $w \in A$, the other midpoint of $G$, $w'$, is in $B$ and at least
	one of the endpoints of $G$ that is adjacent to $w'$ is also in $B$.
	Therefore, $R \cap B = \varnothing$. But $w \in A$, so also $R \cap A=
	\varnothing$, which is impossible. Hence, $w \in B$ and $(A \cup \{v\}, B)$ is
	a $(1,k)$-polar partition of $H$, a contradiction. Then, $H$ is not a minimal
	$(1,k)$-polar obstruction.
\end{proof}

\begin{corollary}
\label{cor:SspiderNOobsmin}
Let $k$ be a nonnegative integer, and let $H$ be a $P_4$-extendible minimal
$(1, k)$-polar obstruction. If $H\not\cong C_5$, then $H$ or its complement is
disconnected.
\end{corollary}

\begin{proof}
	Since $H$ is a $P_4$-extendible graph, we have from \Cref{thm:connChar} that,
	if $H$ and $\overline H$ are connected, then $H$ is either an extension graph
	or a $G$-spider (with nonempty head) for some separable extension graph $G$.
	Nevertheless we have from \Cref{part:ext} of \Cref{rem:extNoObs} and
	\Cref{lem:SspiderNoObs} that this is not the case, so either $H$ or its
	complement is disconnected.
\end{proof}

In the last two results of this section we provide a complete structural
characterization for disconnected $P_4$-extendible minimal $(1,k)$-polar
obstructions. It is worth noticing that statements in
\Cref{lem:extDisc(1k)obsmin,lem:disc(1k)obsminExt} are the same as those in
\Cref{lem:kappa,lem:disc(1k)obsmin}, respectively, except by the obvious
difference of the graph class.

\begin{lemma}
\label{lem:extDisc(1k)obsmin}
Let $t$ be an integer, $t \ge 2$, and for each $i \in \{1, \dots, t\}$, let
$G_i$ be a connected $P_4$-extendible minimal $(1,k_i)$-polar obstruction which
is a $(1, k_i+1)$-polar graph. If $G = G_1 + \dots + G_t$, then $G$ is a minimal
$(1,k)$-polar obstruction if and only if $k = t-1+\sum_{i=1}^t k_i$.
\end{lemma}

\begin{proof}
	Let $k = t-1+\sum_{i=1}^t k_i$. We have from \Cref{lem:basicConnected} that
	$G$ is a minimal $(1,k)$-polar obstruction which is $(1,k+1)$-polar, so we
	just need to show that $G$ is not a minimal $(1,\kappa)$-polar obstruction for
	any $\kappa<k$.

	Let $G_i$ be a connected $P_4$-extendible minimal $(1, k_i)$-polar obstruction
	which is $(1, k_i+1)$-polar. By \Cref{cor:SspiderNOobsmin}, we have that
	either $G_i \cong C_5$ or $\overline{G_i}$ is a disconnected minimal $(k_i,
	1)$-polar obstruction which is a $(k_i+1, 1)$-polar graph. However, it follows
	from \Cref{thm:allDisc(s1)obsmin} that, for any nonnegative integer $\kappa_i$
	such that $\kappa_i < k_i$, $G_i$ contains a proper induced subgraph $G'_i$
	that is both, a $P_4$-extendible minimal $(1,\kappa_i)$-polar obstruction and
	a $(1,\kappa_i+1)$-polar graph. From here on, the proof follows as the proof
	of Lemma 8 in \cite{contrerasDAM281}.
\end{proof}

\begin{lemma}
\label{lem:disc(1k)obsminExt}
Let $k$ be a nonnegative integer.  If $G$ is a disconnected $P_4$-extendible
minimal $(1,k)$-polar obstruction with components $G_1,\dots,G_t$, then there
exist nonnegative integers $k_1, \dots, k_t$ such that for each $i \in \{1,
\dots, t\}$, $G_i$ is a connected minimal $(1,k_i)$-polar obstruction that is a
$(1,k_i+1)$-polar graph, and $\sum_{i=1}^t k_i = k-t+1$. (Notice that $k_i<k$
for any $i \in \{1, \dots, t\}$, and $G$ is a $(1,k+1)$-polar graph.)
\end{lemma}

\begin{proof}
	This is a generalization of Lemma 9 in \cite{contrerasDAM281}, which states
	the same result for cographs. As in the original proof, it is easy to argue
	that each component $G_i$ is a minimal $(1,k_i)$-polar obstruction that is
	$(1,k)$-polar, where $k_i$ is the maximum integer such that any proper induced
	subgraph of $G_i$ is $(1,k_i)$-polar.

	Then, by \Cref{cor:SspiderNOobsmin}, $\overline{G_i}$ is either $C_5$ or a
	disconnected $P_4$-extendible minimal $(k_i, 1)$-polar obstruction which is
	$(k, 1)$-polar. However, it follows from \Cref{thm:allDisc(s1)obsmin} and
	\Cref{part:ext} of \Cref{rem:extNoObs} that $\overline{G_i}$ is a
	$(k_i+1,1)$-polar graph, so $G_i$ is a connected $P_4$-extendible minimal
	$(1,k_i)$-polar obstruction which is a $(1, k_i)$-polar graph. Finally, the
	result follows from \Cref{lem:extDisc(1k)obsmin}.
\end{proof}

The next section is devoted to the most meaningful results of this paper,
including complete characterizations of minimal $(s, 1)$- $(\infty, 1)$- and
$(\infty, \infty)$-polar obstructions on both, $P_4$-sparse and $P_4$-extendible
graphs.


\section{Main results}
\label{sec:main}

In order to analyze the minimal obstructions for polarity in the classes of
$P_4$-sparse and $P_4$-extendible graphs we need a final lemma. Notice that it
holds for general graphs.

\begin{lemma}
\label{lem:charObsmin}
If $G$ is a graph, then $G$ is a disconnected minimal polar obstruction if and
only if $G \cong P_3 + H$ where $H$ is a minimal monopolar obstruction which is
not a minimal polar obstruction.
\end{lemma}

\begin{proof}
	First, assume that $H$ is a minimal $(1,\infty)$-polar obstruction which is
	not a minimal polar obstruction, and let $G = P_3 + H$. Assume for a
	contradiction that $G$ has a polar partition $(A,B)$. Notice that $G[A]$ is
	not an empty graph because $H$ is not a $(1,\infty)$-polar graph. Then $G[A]$
	is completely contained in a component of $G$. Moreover, since any component
	of $G$ is either $P_3$ or a component of $H$, and $G[B]$ is a $P_3$-free
	graph, we have that $A\cap V_H = \varnothing$ so $H$ is a cluster, a
	contradiction. Hence, $G$ is not a polar graph.

	Let $v\in V_G$. If $v\in V_H$, let $(A,B)$ be a $(1, \infty)$-polar partition
	of $H-v$, and let $w\in V_G - V_H$ be a vertex of degree 1. Then $(A', V_G -
	(A'\cup \{v\}))$, where $A' = A\cup \{w\}$, is a $(1,\infty)$-polar partition
	of $G-v$. Now, let $v \in V_G - V_H$. Then, since $H$ is a polar graph and
	$P_3-v$ is a cluster, $G-v$ is a polar graph. Therefore $G$ is a disconnected
	minimal polar obstruction.

	For the converse, assume that $G$ is a disconnected minimal polar obstruction.
	Notice that, if all the components of $G$ are $(1,\infty)$-polar graphs, then
	$G$ is also a $(1,\infty)$-polar graph, so $G$ has a component $H'$ that
	contains a minimal $(1,\infty)$-polar obstruction $H$ as an induced subgraph.
	Notice that by the minimality of $G$, $H$ is a polar graph. In addition, $G$
	has no complete components, so any component of $G$ contains an induced $P_3$,
	and therefore $G$ contains the disjoint union of $P_3$ with a minimal
	$(1,\infty)$-obstruction that is a polar graph ($H$). Together with the
	minimality of $G$, this implies that $G \cong P_3 + H$.
\end{proof}


\subsection{$P_4$-sparse graphs}

The following result provides a  complete recursive construction of $P_4$-sparse
minimal $(s,1)$-polar obstructions.

\begin{theorem} \label{theo:charSp(s1)obsmin}
Let $s$ be an integer, $s \ge 2$. If $G$ is a $P_4$-sparse graph, then $G$ is a
minimal $(s,1)$-polar obstruction if and only if $G$ satisfies exactly one of
the following assertions:
\begin{enumerate}
	\item $G$ is isomorphic to one of the four cographs depicted in
    \Cref{fig:essentials}.

	\item $G$ is isomorphic to some of $2K_{s+1}, K_2 + (K_s\oplus 2K_1)$ or $K_1
    + (K_{s-1}\oplus C_4)$.

	\item The complement of $G$ is disconnected with components $G_1, \dots, G_t$,
    each $G_i$ is a minimal $(1, s_i)$-polar obstruction whose complement is
    different from the graphs in \Cref{fig:essentials}, and $s = t-1+
    \sum_{i=1}^t s_i$.
\end{enumerate}
\end{theorem}

\begin{proof}
  If $G$ is disconnected, it follows from \Cref{thm:allDisc(s1)obsmin} that $G$
  is a minimal $(s,1)$-polar obstruction if and only if $G$ is either a
  $P_4$-sparse graph depicted in \Cref{fig:essentials} (which can easily be
  checked to be a cograph), or it is isomorphic to some of $2K_{s+1}, K_2 +
  (K_s\oplus 2K_1)$ or $K_1 + (K_{s-1}\oplus C_4)$. Otherwise, if $G$ is
  connected, \Cref{cor:spiderNOobsmin} implies that $\overline{G}$ is a
  disconnected $P_4$-sparse minimal $(1, s)$-polar obstruction, and the result
  follows from \Cref{lem:disc(1k)obsmin}.
\end{proof}

\begin{corollary}
\label{cor:psE}
There are exactly nine $P_4$-sparse minimal $(2,1)$-polar obstructions; they are
the graphs $E_1, \dots, E_9$ depicted in
\Cref{fig:essentials,fig:P4ext(21)obsmin}.
\end{corollary}

\begin{figure}[ht!]
\begin{center}
\begin{tikzpicture}

\begin{scope}[yscale=0.75]
\foreach \i in {0,1}
	\foreach \j in {0,1,2}
		\node [vertex] (\i\j) at (\i-0.5,\j-1)[]{};

\foreach \i/\j in {00/11,00/12,01/10,01/12,02/10,02/11,00/01,01/02,10/11,11/12}
  \draw [edge] (\i) to (\j);

\foreach \i/\j in {00/02,12/10}
  \draw [bentE] (\i) to (\j);

\node (n) at (0,-1.75){$E_4 = \overline{3K_2}$};
\end{scope}

\begin{scope}[xshift=3.5cm, yscale=0.75]
\foreach \i in {0,1}
	\foreach \j in {0,1,2}
		\node [vertex] (\i\j) at (\i-0.5,\j-1)[]{};

\foreach \i/\j in {00/01,01/02,10/11,11/12,00/10,00/11,01/10,01/12,02/11,02/12}
  \draw [edge] (\i) to (\j);

\node (n) at (0,-1.8){$E_5 = \overline{K_2+C_4}$};
\end{scope}

\begin{scope}[xshift=7cm, scale=0.75]
\foreach \i in {0,...,3}
	\node [vertex] (\i) at ({(\i*90)}:0.9){};
\node [vertex] (4) at (0,0){};
\node [vertex] (5) at (1,1){};

\foreach \i in {0,...,3}
{
	\draw let \n1={int(mod(\i+1,4)} in [edge] (\i) to (\n1);
	\draw [edge] (\i) to (4);
}

\node (n) at (0,-1.8){$E_6 = K_1 + W_4$};
\end{scope}

\begin{scope}[xshift=0cm, yshift=-3.32cm, scale=0.75, yscale=1]
\foreach \i in {0,...,3}
	\node [vertex] (\i) at ($ (0,-0.5) + ({(\i*90)}:0.9) $){};
\node [vertex] (4) at (-0.6,1.2){};
\node [vertex] (5) at (0.6,1.2){};

\foreach \i in {0,...,3}
	\draw let \n1={int(mod(\i+1,4)} in [edge] (\i) to (\n1);

\foreach \i/\j in {1/3,4/5}
	\draw [edge] (\i) to (\j);

\node (n) at (0,-2.4){$E_8 = K_2+(K_2\oplus \overline{2K_1})$};
\end{scope}

\begin{scope}[xshift=3.5cm, yshift=-3.5cm, scale=0.75]
\foreach \i in {0,...,2}
	\node [vertex] (\i) at ($ (-0.3,0.6) + ({(\i*120)}:0.9) $){};
\foreach \i in {3,...,5}
	\node [vertex] (\i) at ($ (0.3,-0.6) + ({180+(\i*120)}:0.9) $){};

\foreach \i in {0,...,2}
	\draw let \n1={int(mod(\i+1,3)} in [edge] (\i) to (\n1);
\foreach \i in {3,...,5}
	\draw let \n1={int(3+mod(\i+1,3)} in [edge] (\i) to (\n1);

\node (n) at (0,-2.2){$E_9 = 2K_3$};
\end{scope}

\begin{scope}[xshift=7cm, yshift=-3.5cm, scale=0.75]
\foreach \i in {0,...,4}
	\node [vertex] (\i) at ({90+(\i*72)}:1.3){};
\node [vertex] (5) at (-0.4,-0.3){};
\node [vertex] (6) at (0.4,-0.3){};

\foreach \i in {0,...,4}
	\draw let \n1={int(mod(\i+1,5)} in [edge] (\i) to (\n1);
\foreach \i in {0,...,4}
	\foreach \j in {5,6}
		\draw [edge] (\i) to (\j);

\node (n) at (0,-2.2){$E_{13} = \overline{K_2+C_5}$};
\end{scope}

\end{tikzpicture}
\end{center}
\caption{Some minimal $(2,1)$-polar obstructions.}
\label{fig:P4ext(21)obsmin}
\end{figure}
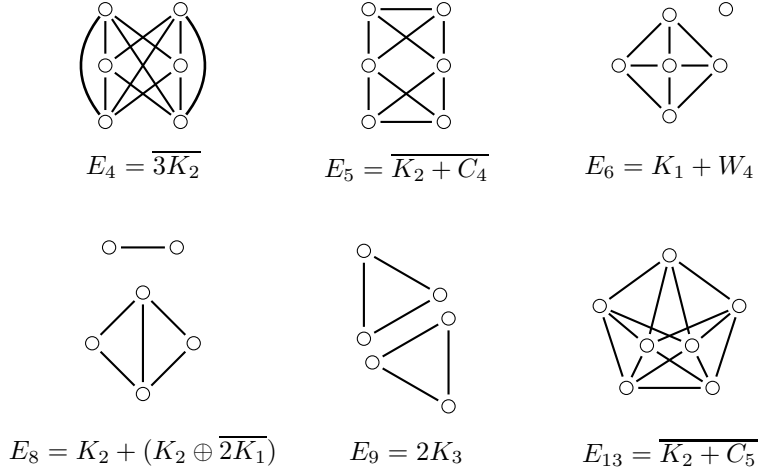

For any hereditary property $\mathcal{P}$ and any graph classes $\mathcal{G}$
and $\mathcal{H}$ such that $\mathcal{G} \subseteq \mathcal{H}$, the set of
minimal $\mathcal{P}$-obstructions in $\mathcal{G}$ clearly is a (possibly
proper) subset of the set of minimal $\mathcal{P}$-obstructions in
$\mathcal{H}$. The class of $P_4$-sparse graphs has been observed to have a
behavior which is very similar to cographs when computing their minimal
obstructions with respect to some hereditary properties. For example,
Hannnebauer \cite{hannnebauerTH'10} proved that every $P_4$-sparse minimal
obstruction for $(k,\ell)$-coloring is a cograph. The following results show
that the same phenomenon occurs when we deal with $(s,1)$-, $(\infty, 1)$-,
$(\infty, \infty)$-polarity.

\begin{theorem}
\label{theo:(s1)sparseCog}
Let $s$ be a nonnegative integer. Any $P_4$-sparse minimal $(s,1)$-polar
obstruction is a cograph.
\end{theorem}

\begin{proof}
  Let $G$ be a $P_4$-sparse minimal $(s,1)$-polar obstruction.   We proceed by
  induction on $s$. The statement is clearly true for $s \le 1$. Let $s \ge 2$.
  It follows from \Cref{cor:spiderNOobsmin} that $G$ is not a spider, hence $G$
  or its complement is disconnected.

  If $G$ is disconnected, it follows from \Cref{thm:allDisc(s1)obsmin} that $G$
  is a cograph. Otherwise, if $\overline G$ is disconnected,
  \Cref{lem:disc(1k)obsmin} implies that any component $H$ of $\overline G$ is a
  $P_4$-sparse minimal $(1,k_i)$-polar obstruction for a nonnegative integer
  $k_i$ with $k_i < k$. Thus, $\overline H$ is a $P_4$-sparse minimal
  $(k_i,1)$-polar obstruction, and by induction hypothesis $\overline H$ (hence
  $H$) is a cograph. Since the disjoint union of cographs is also a cograph,
  $\overline G$ (hence $G$) is a cograph.
\end{proof}

\begin{corollary}
\label{cor:sp(infty1)}
If $G$ is a $P_4$-sparse graph, then $G$ is a minimal $(\infty,1)$-polar
obstruction if and only if $G$ is one of the four cographs depicted in
\Cref{fig:essentials}.
\end{corollary}

\begin{proof}
	Let $G$ be $P_4$-sparse minimal $(\infty,1)$-polar obstruction. Then $G$ is a
	minimal $(s,1)$-polar obstruction for some nonnegative integer $s$. Moreover,
	by \Cref{theo:(s1)sparseCog} we have that $G$ is a cograph minimal
	$(\infty,1)$-polar obstruction. Then, from Theorem 12 in
	\cite{contrerasDAM281} we have that $G$ is isomorphic to one of the cographs
	depicted in \Cref{fig:essentials}. The converse is proved in
	\Cref{lem:discMonopObsmin}.
\end{proof}

\begin{theorem}
If $G$ is a $P_4$-sparse minimal polar obstruction, then $G$ is a cograph.
\end{theorem}

\begin{proof}
  First, assume for a contradiction that $G$ is a spider, say $G = (S, K, R)$.
  Since headless spiders are split graphs, and thus polar graphs, $R$ is not an
  empty set. Moreover, by the minimality of $G$, $G[R]$ admits a polar partition
  $(A, B)$, and then $(A \cup K, B\cup S)$ would be a polar partition of $G$,
  contradicting the choice of $G$. Therefore $G$ is not a spider. Thus, by
  \Cref{thm:connCharSparse}, $G$ or its complement is disconnected. However, in
  both cases \Cref{lem:charObsmin} and \Cref{cor:sp(infty1)} imply that $G$ is a
  cograph.
 \end{proof}


\subsection{$P_4$-extendible graphs}

The following result is analogous to \Cref{theo:charSp(s1)obsmin}; it provides a
complete recursive construction of $P_4$-extendible minimal $(s,1)$-polar
obstructions. Notice that, since $C_5$ is a $P_4$-extendible minimal
$(1,1)$-polar obstruction, there are $P_4$-extendible minimal $(s,1)$-polar
obstructions which are not cographs for each positive integer $s$.

\begin{theorem} \label{theo:charExt(s1)obsmin}
Let $s$ be an integer, $s \ge 2$.   If $G$ is a $P_4$-extendible graph, then $G$
is a minimal $(s,1)$-polar obstruction if and only if $G$ satisfies exactly one
of the following assertions:
\begin{enumerate}
	\item $G$ is isomorphic to one of the seven graphs depicted in
    \Cref{fig:essentials}.

	\item $G$ is isomorphic to some of $2K_{s+1}, K_2 + (K_s \oplus 2K_1)$ or $K_1
    + (K_{s-1}\oplus C_4)$.

	\item The complement of $G$ is disconnected with components $G_1, \dots, G_t$,
    each $G_i$ is a minimal $(1, s_i)$-polar obstruction whose complement is
    different from the graphs in \Cref{fig:essentials}, and $s = t-1+
    \sum_{i=1}^t s_i$.
\end{enumerate}
\end{theorem}

\begin{proof}
  If $G$ is disconnected, it follows from \Cref{thm:allDisc(s1)obsmin} that $G$
  is a minimal $(s,1)$-polar obstruction if and only if $G$ is either a graph
  depicted in \Cref{fig:essentials}, or it is isomorphic to some of $2K_{s+1},
  K_2 + (K_s\oplus 2K_1)$ or $K_1 + (K_{s-1}\oplus C_4)$. Otherwise, if $G$ is
  connected, \Cref{cor:SspiderNOobsmin} and \Cref{part:ext} of
  \Cref{rem:extNoObs} imply that $\overline{G}$ is a disconnected
  $P_4$-extendible minimal $(1, s)$-polar obstruction, and the result follows
  from \Cref{lem:disc(1k)obsminExt}.
\end{proof}

\begin{corollary}
\label{cor:extE}
There are exactly 13 $P_4$-extendible minimal $(2,1)$-polar obstructions; they
are the graphs $E_1, \dots, E_{13}$ depicted in
\Cref{fig:essentials,fig:P4ext(21)obsmin}.
\end{corollary}

Unlike $P_4$-sparse graphs, there are $P_4$-extendible minimal monopolar and
polar obstructions which are not cographs. We give complete lists of such
minimal obstructions in the next results.

\begin{corollary}
\label{cor:P4exCoMonop}
If $G$ is a $P_4$-extendible graph, then $G$ is a minimal $(\infty,1)$-polar
obstruction if and only if $G$ is one of the graphs depicted in
\Cref{fig:essentials}.
\end{corollary}

\begin{proof}
	Let $G$ be a $P_4$-extendible minimal $(\infty,1)$-polar obstruction. Then $G$
	is a minimal $(s,1)$-polar obstruction for some integer $s$, $s \ge 2$. By
	\cref{lem:disc(1k)obsminExt,theo:charExt(s1)obsmin} we conclude that $G$ is
	isomorphic to one of the seven graphs depicted in \Cref{fig:essentials}. The
	converse is proved in \Cref{lem:discMonopObsmin}.
\end{proof}

\begin{theorem}
If $H$ is a $P_4$-extendible minimal polar obstruction, then $H$ or its
complement is the disjoint union of $P_3$ with the complement of one of the
graphs depicted in \Cref{fig:essentials}.
\end{theorem}

\begin{proof}
  First, let assume for obtaining a contradiction that $H$ is a $G$-spider for
  some separable extension $G$, say $H = (S, K, R)$. By
  \Cref{part:P4F,part:C5P5P} of \Cref{rem:extNoObs}, we have that $R \neq
  \varnothing$, and by the minimality of $H$, $H[R]$ admits a polar partition
  $(A, B)$. But, no matter what separable extension $G$ is, its midpoints induce
  a complete multipartite graph while its endpoints induce a cluster, so $(A\cup
  K, B\cup S)$ is a polar partition of $H$, contradicting the assumption that
  $H$ was a $G$-spider. Thus, by \Cref{thm:connChar}, either $H$ or its
  complement is disconnected, and the result follows from
  \Cref{lem:charObsmin,cor:P4exCoMonop}.
\end{proof}


\section{Conclusions}
\label{sec:conclusions}

In the present work we generalize some results related to hereditary properties
in cographs, providing similar results for two superclasses of $P_4$-free
graphs, namely $P_4$-sparse and $P_4$-extendible graphs. The following five
theorems summarize the main contributions of this paper. Let $\mathcal{G}$ be
any subclass of either $P_4$-extendible or $P_4$-sparse graphs which is closed
under both graph complements and induced subgraphs.

\begin{theorem}
Let $G$ be a graph in the class $\mathcal{G}$. Then $G$ is a minimal unipolar
obstruction if and only if $G \in \{2P_3, K_{2,3}, C_5\}$.
\end{theorem}

\begin{theorem}
\label{theo:main}
Let $G$ be a graph in the class $\mathcal{G}$, and let $s$ be an integer, $s \ge
2$. Then $G$ is a minimal $(s,1)$-polar obstruction if and only if $G$ satisfies
exactly one of the following assertions:
\begin{enumerate}
	\item $G$ is isomorphic to one of the graphs depicted in
    \Cref{fig:essentials}.

	\item $G$ is isomorphic to some of $2K_{s+1}, K_2 + (K_s \oplus 2K_1)$ or $K_1
    + (K_{s-1}\oplus C_4)$.

	\item The complement of $G$ is disconnected with components $G_1, \dots, G_t$,
    each $G_i$ is a minimal $(1, s_i)$-polar obstruction whose complement is
	  different from the graphs in \Cref{fig:essentials}, and $s = t-1 +
    \sum_{i=1}^t s_i$.
\end{enumerate}
\end{theorem}

\begin{theorem}
Let $G$ be a graph in the class $\mathcal{G}$. Then $G$ is a minimal monopolar
obstruction if and only if $\overline G$ is one of the graphs depicted in
\Cref{fig:essentials}.
\end{theorem}

\begin{theorem}
Let $G$ be a graph in the class $\mathcal{G}$. Then, $G$ is a minimal polar
obstruction if and only if either $G$ or its complement is the join of
$\overline{P_3}$ with one of the graphs depicted in \Cref{fig:essentials}.
\end{theorem}

\begin{theorem}
Any hereditary property $\mathcal P$ on the class $\mathcal{G}$ has only a finite
number of minimal $\mathcal P$-obstructions.
\end{theorem}

Throughout this work we showed that any $P_4$-sparse minimal obstruction for
unipolarity, monopolarity, polarity, and $(s,1)$-polarity is a cograph. In
addition, Hannnebauer \cite{hannnebauerTH'10} showed the following interesting
result that generalize its analogue for cographs, which was previously proved in
\cite{federGTP}.

\begin{theorem} [\cite{hannnebauerTH'10}]
\label{theo:P4sp(s+1)(k+1)}
Let $H$ be a $P_4$-sparse minimal $(s,k)$-polar obstruction. Then $H$ has at
most $(s+1)(k+1)$ vertices.
\end{theorem}

The observations above make us propose the following questions.

\begin{problem}
\label{prob:obsminSparseRcog}
For any positive integers $s$ and $k$, is every $P_4$-sparse minimal
$(s,k)$-polar obstruction a cograph?
\end{problem}

\begin{problem}
  Can we establish an $O(sk)$ upper bound for the order of the
  $P_4$-extendible minimal $(s,k)$-polar obstructions?
\end{problem}

It was independently shown in \cite{bravoDAM159} and \cite{hannnebauerTH'10}
that any $P_4$-sparse minimal obstruction for $(k,\ell)$-coloring is a cograph
too, so we propose the following problem generalizing
\Cref{prob:obsminSparseRcog}.

\begin{problem}
Which hereditary properties $\mathcal{P}$ satisfy that every $P_4$-sparse
minimal $\mathcal{P}$-obstruction is a cograph?
\end{problem}

With the help of an interesting graph operation called partial complementation,
Hell, Hern\'andez-Cruz and Linhares-Sales \cite{hellDAM261} gave the complete
list of cograph minimal $(2,2)$-polar obstructions. In a work in progress, we
provide analogous results for $P_4$-sparse and $P_4$-extendible graphs, as well
as efficient algorithms for finding maximal unipolar, monopolar, and polar
subgraphs on these families. Such algorithms are based on the unique tree
representations for the mentioned classes and they generalize those given in
\cite{ekimDAM156b} for cographs.

As a future line of work, we propose to extend the results in this paper to more
general graph classes having few induced $P_4$'s, for instance, $P_4$-tidy
graphs or extended $P_4$-laden graphs, any of which contains both $P_4$-sparse
and $P_4$-extendible graphs. Another line of work is to characterize some other
hereditary properties on cograph superclasses (see Figure \ref{fig:relations})
by their sets of minimal obstructions. For example, it remains unknown whether
the $P_4$-extendible minimal $(k,\ell)$-obstructions admit a simple structural
characterization as their analogous in cographs and $P_4$-sparse graphs.


\end{document}